\documentclass[a4paper]{amsproc}
\usepackage{etoolbox,fontspec}
\makeatletter
\newdimen\markdownclass@dateskip
\let\markdownclass@dateprefix\@empty
\makeatother
\input{tex/preamble_precompiled.tex}

\makeatletter
\RequirePackage{verified-badges}
\DeclareFormalizationRoot[rocq]{rocq}{\detokenize{https://github.com/TurtlePU/pts-rk/blob/d22053fa22d6070c8a14af4da679e439dd012825/theories/}}
\DeclareFormalizationRoot[rocq]{rocq-ai}{\detokenize{https://github.com/TurtlePU/pts-rk/blob/8e3a7b67b571f6ab3fb7cc67e9dc7c0a214351ce/theories/}}
\title{A sound pure type system with ranks as sorts and Type-in-Type}
\author[ ]{Pavel Sokolov\raisebox{-0.45ex}{\scriptsize\orcidlink{0009-0001-6922-2870}}\kern-0.4em\textsuperscript{1} ‹\protect\url{ppsokolov@hse.ru}›}
\author[ ]{Alexander Kuklev\raisebox{-0.45ex}{\scriptsize\orcidlink{0009-0004-1019-4205}}\kern-0.4em\textsuperscript{2} ‹\protect\url{a@kuklev.com}›}
\affil[1]{\small HSE University, Laboratory of Theoretical Computer Science}
\affil[2]{\small Radboud Univ. Nijmegen, Department of Software Science}
\ifdefined\subjclass\subjclass[2020]{03B38, 03D15, 03B70, 03B40, 68N18}\fi
\date{\markdownclass@dateprefix October 16, 2026}

\makeatother
\begin{document}
\maketitle
\begingroup
\catcode`\%=12\relax
\catcode`\#=12\relax
\markdownRendererDocumentBegin
\begin{abstract}We present a dependently typed pure type system\markdownRendererSoftLineBreak
{}that contains a circular axiom \markdownRendererInlineMath{s : s} for a\markdownRendererSoftLineBreak
{}certain sort \markdownRendererInlineMath{s} and is nonetheless strongly\markdownRendererSoftLineBreak
{}normalizing.\markdownRendererSoftLineBreak
{}Girard's paradox is avoided because our sorts are\markdownRendererSoftLineBreak
{}not closed under a \markdownRendererInlineMath{Π}-former:\markdownRendererSoftLineBreak
{}function types are placed in sorts according to\markdownRendererSoftLineBreak
{}their \markdownRendererEmphasis{rank}.\markdownRendererSoftLineBreak
{}Turns out that this system has a unique set of\markdownRendererSoftLineBreak
{}characteristics: it has similar normalization\markdownRendererSoftLineBreak
{}behaviour as STLC (so complexity of normalization\markdownRendererSoftLineBreak
{}is \textsf{TOWER}-complete) while at the same time\markdownRendererSoftLineBreak
{}allowing non-uniform representation of arbitrary\markdownRendererSoftLineBreak
{}elementary recursive functions; furthermore,\markdownRendererSoftLineBreak
{}it has a nontrivial set-theoretical model in\markdownRendererSoftLineBreak
{}ZF+GCH and allows to carry out nontrivial proof\markdownRendererSoftLineBreak
{}developments.\markdownRendererSoftLineBreak
{}Base normalization proof is surprisingly simple\markdownRendererSoftLineBreak
{}(modulo general metatheoretical properties of pure\markdownRendererSoftLineBreak
{}type systems) and has been fully formalized in\markdownRendererSoftLineBreak
{}vanilla Rocq (formerly Coq) without additional\markdownRendererSoftLineBreak
{}axioms.\end{abstract}\markdownRendererInterblockSeparator
{}\markdownRendererSectionBegin
\markdownRendererHeadingOne{Introduction}\markdownRendererInterblockSeparator
{}Pure type systems (PTS) are a family of minimalistic\markdownRendererSoftLineBreak
{}typed \markdownRendererInlineMath{λ}-calculi introduced independently by\markdownRendererSoftLineBreak
{}Berardi \markdownRendererCite{1}+{}{}{berardi1988towards} and Terlouw \markdownRendererCite{1}+{}{}{terlouw1989gstt}.\markdownRendererSoftLineBreak
{}Their abstract grammar includes only variables\markdownRendererSoftLineBreak
{}\markdownRendererInlineMath{x ∈ \mathsf{Var}}, abstraction \markdownRendererInlineMath{λ(x : T).u},\markdownRendererSoftLineBreak
{}application \markdownRendererInlineMath{u\,t}, dependent function types\markdownRendererSoftLineBreak
{}\markdownRendererInlineMath{Π(x : T).U} (where \markdownRendererInlineMath{U} might depend on \markdownRendererInlineMath{x}) and\markdownRendererSoftLineBreak
{}\markdownRendererEmphasis{sorts} \markdownRendererInlineMath{s ∈ \mathcal{S}} that act as “types of\markdownRendererSoftLineBreak
{}types”.\markdownRendererInterblockSeparator
{}\begin{figure}[ht]\begin{bnf}
  $x$ : \textsf{Variable} :in: $X$;;
  $s$ : \textsf{Sort} :in: $S$;;
  $T$, $U$, $t$, $u$ : \textsf{Term} ::=
  | $x$ : variable
  | $s$ : sort as type
  | $Π (x \colon T).\,U$
      : type of dependent functions
  | $λ (x \colon T).\,u$ : anonymous function
  | $u\,t$ : function application
  ;;
\end{bnf}\caption{Grammar of pure type systems}\end{figure}\markdownRendererInterblockSeparator
{}Common syntax conventions to make it more\markdownRendererSoftLineBreak
{}lightweight are the following:\markdownRendererSoftLineBreak
{}application associates to the left,\markdownRendererSoftLineBreak
{}so \markdownRendererInlineMath{a\,b\,c} is understood as \markdownRendererInlineMath{(a\,b)\,c};\markdownRendererSoftLineBreak
{}in a sequence of abstractions,\markdownRendererSoftLineBreak
{}all \markdownRendererInlineMath{λ} symbols except the first can be omitted;\markdownRendererSoftLineBreak
{}\markdownRendererInlineMath{T → U} is an alias for \markdownRendererInlineMath{Π (x : T).U} (where \markdownRendererInlineMath{U}\markdownRendererSoftLineBreak
{}does not depend on \markdownRendererInlineMath{x}); \markdownRendererInlineMath{( → )} is\markdownRendererSoftLineBreak
{}right-associative, so \markdownRendererInlineMath{T → U → V} is understood as\markdownRendererSoftLineBreak
{}\markdownRendererInlineMath{T → (U → V)}.\markdownRendererSoftLineBreak
{}In addition, even though abstractions are always\markdownRendererSoftLineBreak
{}annotated with a type of an argument,\markdownRendererSoftLineBreak
{}we are going to omit it on paper whenever they are\markdownRendererSoftLineBreak
{}clear from the context.\markdownRendererSoftLineBreak
{}Also note that there is no syntactical distinction\markdownRendererSoftLineBreak
{}between terms and types — they are defined by a\markdownRendererSoftLineBreak
{}single grammar; an expression can act both as a\markdownRendererSoftLineBreak
{}term and as a type depending on a side where it\markdownRendererSoftLineBreak
{}stands next to the colon in the typing judgement\markdownRendererSoftLineBreak
{}\markdownRendererInlineMath{Γ ⊢ \_ : \_}.\markdownRendererParagraphSeparator
{}A single PTS is determined by two relations:\markdownRendererInterblockSeparator
{}\markdownRendererUlBegin
\markdownRendererUlItem \markdownRendererEmphasis{axiom} relation\markdownRendererSoftLineBreak
{}\markdownRendererInlineMath{\mathsf{Ax} ⊆ \mathcal{S} × \mathcal{S}}\markdownRendererSoftLineBreak
{}determines “to which sort(s) a given sort\markdownRendererSoftLineBreak
{}belongs as an element”; instead of\markdownRendererSoftLineBreak
{}\markdownRendererInlineMath{(s, s') ∈ \mathsf{Ax}} it is customary to write\markdownRendererSoftLineBreak
{}simply \markdownRendererInlineMath{s : s'}.\markdownRendererUlItemEnd 
\markdownRendererUlItem \markdownRendererEmphasis{rule} relation\markdownRendererSoftLineBreak
{}\markdownRendererInlineMath{\mathsf{Rule} ⊆ \mathcal{S}^2 × \mathcal{S}}\markdownRendererSoftLineBreak
{}places types of functions \markdownRendererInlineMath{Π (x : T).U} into\markdownRendererSoftLineBreak
{}sorts accounting for the sorts where \markdownRendererInlineMath{T} and \markdownRendererInlineMath{U}\markdownRendererSoftLineBreak
{}are situated.\markdownRendererUlItemEnd 
\markdownRendererUlEnd \markdownRendererInterblockSeparator
{}\begin{figure}[h]\begin{mathpar}
  \inferrule[Ax]{ }{Γ ⊢ s : s'}
      (s, s') ∈ \mathsf{Ax}
  \and
  \inferrule[Rule]
    {Γ ⊢ T : s_T \\ Γ, x : T ⊢ U : s_U}
    {Γ ⊢ Π (x : T).\, U : s}
        (s_T, s_U, s) ∈ \mathsf{Rule}
  \and
  \inferrule[Var]{ }{Γ ⊢ x : T}
      x : T ∈ Γ
  \and
  \inferrule[Abs]
    {Γ ⊢ T : s_T \\ Γ, x : T ⊢ u : U}
    {Γ ⊢ λ (x : T). u : Π(x : T). U}
  \and
  \inferrule[App]
    {Γ ⊢ u : Π(x : T). U \\ Γ ⊢ t : T}
    {Γ ⊢ u\,t : U[t/x]}
  \and
  \inferrule[Conv]
    {Γ ⊢ t : T \\ Γ ⊢ U : s \\ T =_β U}
    {Γ ⊢ t : U}
\end{mathpar}\caption{Typing rules of pure type systems}\label{typing}\end{figure}\markdownRendererInterblockSeparator
{}Systems from Barendregt's lambda cube \markdownRendererCite{2}+{}{}{barendregt1991introduction}+{}{}{barendregt1992lambda}\markdownRendererSoftLineBreak
{}fit naturally into the framework: for example,\markdownRendererSoftLineBreak
{}simply typed lambda calculus (STLC) \markdownRendererCite{1}+{}{}{church1940formulation}\markdownRendererSoftLineBreak
{}can be viewed as a pure type system with a single\markdownRendererSoftLineBreak
{}sort \markdownRendererInlineMath{∗}, no axioms and a single rule \markdownRendererInlineMath{(∗, ∗, ∗)}.\markdownRendererSoftLineBreak
{}In general, however, pure type systems are\markdownRendererSoftLineBreak
{}dependently typed since \markdownRendererInlineMath{Π} allows the codomain\markdownRendererSoftLineBreak
{}type to depend on the value of an argument —\markdownRendererSoftLineBreak
{}Coquand's Calculus of Constructions (CoC) \markdownRendererCite{1}+{}{}{coquand1988calculus}\markdownRendererSoftLineBreak
{}sits comfortably at the top of lambda cube and is\markdownRendererSoftLineBreak
{}modelled as a pure type system with two sorts\markdownRendererSoftLineBreak
{}\markdownRendererInlineMath{\{ ∗, □\}}, a single axiom \markdownRendererInlineMath{∗ : □} and four rules\markdownRendererSoftLineBreak
{}corresponding to simple functions \markdownRendererInlineMath{(∗, ∗, ∗)},\markdownRendererSoftLineBreak
{}polymorphic functions \markdownRendererInlineMath{(□, ∗, ∗)},\markdownRendererSoftLineBreak
{}type operators \markdownRendererInlineMath{(□, □, □)} and dependent types\markdownRendererSoftLineBreak
{}\markdownRendererInlineMath{(∗, □, □)}.\markdownRendererParagraphSeparator
{}Both STLC and CoC are strongly normalizing and\markdownRendererSoftLineBreak
{}sound with respect to their set-theoretical\markdownRendererSoftLineBreak
{}semantics \markdownRendererCite{3}+{}{}{tait1967intensional}+{}{}{coquand1990kripke}+{}{}{geuvers1991modular},\markdownRendererSoftLineBreak
{}but it is not true for arbitrary PTS.\markdownRendererInterblockSeparator
{}\begin{theorem*}[Girard'72 \cite{girard1972interpretation, hurkens1995simplification}]The following PTS called System U is not\markdownRendererSoftLineBreak
{}normalizing:\begin{mathpar}
  \begin{array}{c}
    \mathcal{S} = \{ ∗, □, △\},
    \quad
    \mathsf{Ax} = \{(∗, □), (□, △)\},
    \\
    \mathsf{Rule} = \{
      (∗, ∗, ∗), (□, ∗, ∗), (△, ∗, ∗),
      (□, □, □), (△, □, □)
    \}
  \end{array}
\end{mathpar}\markdownRendererInterblockSeparator
{}Furthermore, the offending term has type\markdownRendererSoftLineBreak
{}\markdownRendererInlineMath{Π (P : ∗).P}.\end{theorem*}\markdownRendererInterblockSeparator
{}This result is dubbed “Girard's paradox” since it\markdownRendererSoftLineBreak
{}renders System U inconsistent with respect to its\markdownRendererSoftLineBreak
{}intended logical semantics: inhabitedness of\markdownRendererSoftLineBreak
{}\markdownRendererInlineMath{Π (P : ∗).P} would mean that every proposition is\markdownRendererSoftLineBreak
{}provable.\markdownRendererSoftLineBreak
{}Thus, any system where System U embeds is\markdownRendererSoftLineBreak
{}inconsistent, too; in particular,\markdownRendererSoftLineBreak
{}an original version \markdownRendererCite{1}+{}{}{martinlof1971theory} of\markdownRendererSoftLineBreak
{}Martin-Löf's Intuitionistic Type Theory \markdownRendererCite{1}+{}{}{martinlof1984intuitionistic}\markdownRendererSoftLineBreak
{}embeds the following PTS known as “Type-in-Type”\markdownRendererSoftLineBreak
{}which is inconsistent due to Girard's paradox:\begin{mathpar}
  \mathcal{S} = \{ ∗ \}, \quad \mathsf{Ax} = \{ ∗ : ∗ \}, \quad
  \mathsf{Rule} = \{(∗, ∗, ∗)\}.
\end{mathpar}\markdownRendererInterblockSeparator
{}For a pure type system with usual set-theoretical\markdownRendererSoftLineBreak
{}semantics, it is natural to expect that a sort of\markdownRendererSoftLineBreak
{}\markdownRendererEmphasis{small types} \markdownRendererInlineMath{∗} is closed under \markdownRendererInlineMath{Π},\markdownRendererSoftLineBreak
{}thus we cannot forbid the rule \markdownRendererInlineMath{(∗, ∗, ∗)}; then,\markdownRendererSoftLineBreak
{}a circular axiom \markdownRendererInlineMath{∗ : ∗} has to go.\markdownRendererSoftLineBreak
{}Long story short, it motivates an introduction of\markdownRendererSoftLineBreak
{}a \markdownRendererEmphasis{universe hierarchy} \markdownRendererCite{2}+{}{}{martinlof1975intuitionistic}+{}{}{harper1991universes}\markdownRendererSoftLineBreak
{}\markdownRendererInlineMath{\mathcal{S} = \{\mathcal{U}_k\}_{k ∈ ℕ}} with\markdownRendererSoftLineBreak
{}axioms \markdownRendererInlineMath{\mathcal{U}_k : \mathcal{U}_{k + 1}} and\markdownRendererSoftLineBreak
{}rules\markdownRendererSoftLineBreak
{}\markdownRendererInlineMath{(\mathcal{U}_k, \mathcal{U}_l, \mathcal{U}_{k ∨ l})}\markdownRendererSoftLineBreak
{}for each \markdownRendererInlineMath{k, l}; when interpreted in sets,\markdownRendererSoftLineBreak
{}\markdownRendererInlineMath{⟦ \mathcal{U}_k ⟧} is taken to be the \markdownRendererInlineMath{k}-th\markdownRendererSoftLineBreak
{}inaccessible cardinal \markdownRendererCite{2}+{}{}{werner1997sets}+{}{}{aczel1998relating}.\markdownRendererParagraphSeparator
{}Observing this, we took an interest in the following\markdownRendererSoftLineBreak
{}two questions:\markdownRendererInterblockSeparator
{}\markdownRendererUlBegin
\markdownRendererUlItem If we were to drop the rule \markdownRendererInlineMath{(∗, ∗, ∗)} and keep\markdownRendererSoftLineBreak
{}the axiom \markdownRendererInlineMath{∗ : ∗} instead, can we find a\markdownRendererSoftLineBreak
{}nontrivial consistent pure type system?\markdownRendererUlItemEnd 
\markdownRendererUlItem Is there a consistent pure type system with\markdownRendererSoftLineBreak
{}total relations \markdownRendererInlineMath{\mathsf{Ax}} and\markdownRendererSoftLineBreak
{}\markdownRendererInlineMath{\mathsf{Rule}} whose natural set-theoretical\markdownRendererSoftLineBreak
{}semantics do not involve the assumption of\markdownRendererSoftLineBreak
{}inaccessible cardinals?\markdownRendererUlItemEnd 
\markdownRendererUlEnd \markdownRendererInterblockSeparator
{}Turns out that, for both questions,\markdownRendererSoftLineBreak
{}the answer is “yes”.\markdownRendererSoftLineBreak
{}Furthermore, both yes-s are witnessed by a single\markdownRendererSoftLineBreak
{}PTS, which we dub \PTSrk{}:\begin{mathpar}
  \mathcal{S} = \{\rk{r}\}_{r ∈ ℕ}
  \quad
  \mathsf{Ax} = \{(\rk{r}, \rk{0})\}_{r ∈ ℕ}
  \quad
  \mathsf{Rule} = \{(\rk{r}, \rk{s},
          \rk{(r + 1) ∨ s})\}_{r, s ∈ ℕ}
\end{mathpar}\markdownRendererInterblockSeparator
{}(Naming “\markdownRendererInlineMath{\mathsf{rk}}” comes from \markdownRendererEmphasis{rank},\markdownRendererSoftLineBreak
{}since rules compute precisely the rank of\markdownRendererSoftLineBreak
{}functions, which is the maximal nesting of \markdownRendererInlineMath{Π} in\markdownRendererSoftLineBreak
{}the domain type.)\markdownRendererParagraphSeparator
{}Indeed, there is a circular axiom\markdownRendererSoftLineBreak
{}\markdownRendererInlineMath{\rk{0} : \rk{0}}, and System U is not embeddable\markdownRendererSoftLineBreak
{}since \markdownRendererInlineMath{\rk{r}} is not closed under \markdownRendererInlineMath{Π} for any\markdownRendererSoftLineBreak
{}\markdownRendererInlineMath{r}.\markdownRendererSoftLineBreak
{}Also, as we will see in Section \markdownRendererLink{\markdownRendererHash{}rk-sn}{\markdownRendererHash{}rk-sn}{#rk-sn}{},\markdownRendererSoftLineBreak
{}our \PTSrk{} is strongly normalizing despite an\markdownRendererSoftLineBreak
{}intuition that \markdownRendererInlineMath{\rk{r} : \rk{0}} might “feel\markdownRendererSoftLineBreak
{}wrong”.\markdownRendererSoftLineBreak
{}Modulo general metatheory of pure type systems\markdownRendererSoftLineBreak
{}laid out in Section \markdownRendererLink{\markdownRendererHash{}pts-metatheory}{\markdownRendererHash{}pts-metatheory}{#pts-metatheory}{},\markdownRendererSoftLineBreak
{}a normalization proof turned out to be\markdownRendererSoftLineBreak
{}suspiciously simple, even more so in comparison\markdownRendererSoftLineBreak
{}with normalization proofs of other dependently\markdownRendererSoftLineBreak
{}typed systems \markdownRendererCite{3}+{}{}{geuvers1991modular}+{}{}{barras1999autovalidation}+{}{}{abel2018decidability}.\markdownRendererSoftLineBreak
{}To dispel our own doubts, we formalized a proof of\markdownRendererSoftLineBreak
{}normalization in Rocq (formerly Coq) without using\markdownRendererSoftLineBreak
{}any additional axioms, yielding a fully\markdownRendererSoftLineBreak
{}constructive proof; our journey in formalization\markdownRendererSoftLineBreak
{}is summarized in Section \markdownRendererLink{\markdownRendererHash{}formalization}{\markdownRendererHash{}formalization}{#formalization}{} and the\markdownRendererSoftLineBreak
{}whole development is available at \markdownRendererLink{https://github.com/TurtlePU/pts-rk}{https://github.com/TurtlePU/pts-rk}{https://github.com/TurtlePU/pts-rk}{}.\markdownRendererParagraphSeparator
{}Having completed the proof, we noticed that it\markdownRendererSoftLineBreak
{}exposes a primitive recursive procedure to\markdownRendererSoftLineBreak
{}normalize terms of \PTSrk{}; we dug further and\markdownRendererSoftLineBreak
{}discovered that, actually, there is an intimate\markdownRendererSoftLineBreak
{}connection between a class of elementary recursive\markdownRendererSoftLineBreak
{}functions \markdownRendererCite{2}+{}{}{kalmar1943}+{}{}{rose1984subrecursion} and\markdownRendererSoftLineBreak
{}our system, as seen in Sections \markdownRendererLink{\markdownRendererHash{}complexity}{\markdownRendererHash{}complexity}{#complexity}{}-\markdownRendererLink{\markdownRendererHash{}computational}{\markdownRendererHash{}computational}{#computational}{}.\markdownRendererSoftLineBreak
{}As a byproduct of our study on complexity,\markdownRendererSoftLineBreak
{}we built a sound set-theoretic model for \PTSrk{}\markdownRendererSoftLineBreak
{}(see Section \markdownRendererLink{\markdownRendererHash{}rk-model}{\markdownRendererHash{}rk-model}{#rk-model}{}) which assumes\markdownRendererSoftLineBreak
{}\markdownRendererEmphasis{Generalized Continuum Hypothesis} \markdownRendererCite{4}+{}{}{jech2003set}+{}{}{hausdorff1908}+{}{}{godel1938consistency}+{}{}{godel1940consistency}\markdownRendererSoftLineBreak
{}which is an unusual pick for dependently typed\markdownRendererSoftLineBreak
{}systems where \markdownRendererEmphasis{inaccessible cardinals} are a more\markdownRendererSoftLineBreak
{}common assumption \markdownRendererCite{2}+{}{}{werner1997sets}+{}{}{lee2011proof}.\markdownRendererSoftLineBreak
{}Before concluding, we provide a sneak peek into\markdownRendererSoftLineBreak
{}what it is like to write proofs in \PTSrk{} and\markdownRendererSoftLineBreak
{}provide a full interpretation of\markdownRendererSoftLineBreak
{}Tarski–Mostowski–Robinson arithmetic \markdownRendererInlineMath{R} (Section \markdownRendererLink{\markdownRendererHash{}proof-system}{\markdownRendererHash{}proof-system}{#proof-system}{}).\markdownRendererInterblockSeparator
{}\markdownRendererSectionBegin
\markdownRendererHeadingTwo{Use of artificial intelligence}\markdownRendererInterblockSeparator
{}At various points in our research,\markdownRendererSoftLineBreak
{}we interacted with GPT-6 Astra \markdownRendererCite{1}+{}{}{gpt6astra2026} to\markdownRendererSoftLineBreak
{}prove something for us.\markdownRendererSoftLineBreak
{}Sometimes it was a rehash of previously\markdownRendererSoftLineBreak
{}established results; sometimes it presented\markdownRendererSoftLineBreak
{}something new.\markdownRendererSoftLineBreak
{}In case we used the ideas hinted by LLM,\markdownRendererSoftLineBreak
{}we explicitly state so right next to the proof;\markdownRendererSoftLineBreak
{}even so, the text is completely human written\markdownRendererSoftLineBreak
{}since we wanted to make sure we completely\markdownRendererSoftLineBreak
{}understood the results (in addition,\markdownRendererSoftLineBreak
{}generated proofs always required additional\markdownRendererSoftLineBreak
{}explanations or a complete rewrite to make them\markdownRendererSoftLineBreak
{}understandable and simpler to read).\markdownRendererSoftLineBreak
{}We also used Claude Fable 5.1 \markdownRendererCite{1}+{}{}{claude_fable}\markdownRendererSoftLineBreak
{}to identify additional points of reference.\markdownRendererInterblockSeparator
{}
\markdownRendererSectionEnd \markdownRendererSectionBegin
\markdownRendererHeadingTwo{Formalization status}\markdownRendererInterblockSeparator
{}Every result up to the proof of strong\markdownRendererSoftLineBreak
{}normalization for \PTSrk{} was formally verified\markdownRendererSoftLineBreak
{}in Rocq using handwritten proofs.\markdownRendererSoftLineBreak
{}Such proofs are marked with a \proof{rocq-ai:https://github.com/TurtlePU/pts-rk};\markdownRendererSoftLineBreak
{}each icon is a hyperlink to the statement and\markdownRendererSoftLineBreak
{}proof in Rocq with the version pinned at the moment of\markdownRendererSoftLineBreak
{}submission.\markdownRendererParagraphSeparator
{}Results after that were formalized in Rocq as\markdownRendererSoftLineBreak
{}well, albeit not by us but rather by GPT-6 Astra.\markdownRendererSoftLineBreak
{}Such results are marked with \formalized{rocq-ai:https://github.com/TurtlePU/pts-rk/pull/1};\markdownRendererSoftLineBreak
{}just as with handwritten results, each icon is a\markdownRendererSoftLineBreak
{}hyperlink to the statement and proof in Rocq with\markdownRendererSoftLineBreak
{}the version pinned at the moment of submission.\markdownRendererSoftLineBreak
{}Where only a component of a paper claim is\markdownRendererSoftLineBreak
{}checked, a note explaining what exactly is\markdownRendererSoftLineBreak
{}formalized accompanies the badge.\markdownRendererInterblockSeparator
{}
\markdownRendererSectionEnd 
\markdownRendererSectionEnd \markdownRendererSectionBegin
\markdownRendererHeaderAttributeContextBegin
\markdownRendererAttributeIdentifier{pts-metatheory}\markdownRendererHeadingOne{Metatheory of pure type systems}\markdownRendererHeaderAttributeContextEnd{}\markdownRendererInterblockSeparator
{}For ease of presentation of our results in this\markdownRendererSoftLineBreak
{}and following sections, we assume a Barendregt\markdownRendererSoftLineBreak
{}convention \markdownRendererCite{1}+{}{}{barendregt1984lambda} that all bound\markdownRendererSoftLineBreak
{}variables are distinct from free variables and\markdownRendererSoftLineBreak
{}from each other; furthermore,\markdownRendererSoftLineBreak
{}\markdownRendererInlineMath{α}-equivalence (renaming of bound variables) is\markdownRendererSoftLineBreak
{}assumed as a trivial equality on terms.\markdownRendererSoftLineBreak
{}Nevertheless, our formalization takes variable\markdownRendererSoftLineBreak
{}treatment seriously; see Section \markdownRendererLink{\markdownRendererHash{}formalization}{\markdownRendererHash{}formalization}{#formalization}{}\markdownRendererSoftLineBreak
{}for more details.\markdownRendererInterblockSeparator
{}\markdownRendererSectionBegin
\markdownRendererHeadingTwo{PTS as a rewriting system}\markdownRendererInterblockSeparator
{}In Fig.\markdownRendererSoftLineBreak
{}\markdownRendererLink{\markdownRendererHash{}typing}{\markdownRendererHash{}typing}{#typing}{}, a rule of \markdownRendererEmphasis{conversion} (\textsc{Conv})\markdownRendererSoftLineBreak
{}mentions \markdownRendererEmphasis{definitional equality} of types\markdownRendererSoftLineBreak
{}\markdownRendererInlineMath{T =_β U}.\markdownRendererSoftLineBreak
{}It is defined as a reflexive, symmetric,\markdownRendererSoftLineBreak
{}transitive closure of a one-step reduction.\markdownRendererInterblockSeparator
{}\begin{definition}[One-step reduction]Relation of a one-step reduction \markdownRendererInlineMath{t →_β t'} is the\markdownRendererSoftLineBreak
{}smallest relation satisfying the following rules:\begin{mathpar}
  \inferrule{T →_β T'}
      {Π (x : T).U →_β Π (x : T').U}
  \and
  \inferrule{U →_β U'}
      {Π (x : T).U →_β Π (x : T).U'}
  \and
  \inferrule{T →_β T'}
      {λ (x : T).t →_β λ (x : T').t}
  \and
  \inferrule{t →_β t'}
      {λ (x : T).t →_β λ (x : T).t'}
  \and
  \inferrule{t →_β t'}{t\,u →_β t'\,u}
  \and
  \inferrule{u →_β u'}{t\,u →_β t\,u'}
  \and
  \inferrule[Beta]{ }
      {(λ (x : T).t)\,u →_β t[u/x]}
\end{mathpar}\end{definition}\markdownRendererInterblockSeparator
{}In essence, actual reduction happens only in rule\markdownRendererSoftLineBreak
{}\textsc{Beta}; other rules simply state that we can\markdownRendererSoftLineBreak
{}freely reduce any subterm.\markdownRendererParagraphSeparator
{}The following result will be greatly useful in our\markdownRendererSoftLineBreak
{}development.\markdownRendererSoftLineBreak
{}It was proven many times over for different\markdownRendererSoftLineBreak
{}similar systems \markdownRendererCite{2}+{}{}{church1936properties}+{}{}{barendregt1992lambda};\markdownRendererSoftLineBreak
{}we briefly revisit the proof since previous\markdownRendererSoftLineBreak
{}presentations usually did not allow reductions in\markdownRendererSoftLineBreak
{}the type of argument \markdownRendererInlineMath{T} in \markdownRendererInlineMath{λ (x : T).t}.\markdownRendererInterblockSeparator
{}\begin{theorem}[Characterization of $( =_\beta )$ \proof{rocq:DefEq.v\#L9}]In an arbitrary PTS,\begin{mathpar}
  t =_β u \quad ⇔ \quad
  ∃ v.\;t →_β^∗ v ∧ u →_β^∗ v,
\end{mathpar}\markdownRendererInterblockSeparator
{}where \markdownRendererInlineMath{( →_β^∗ )} is the reflexive transitive closure\markdownRendererSoftLineBreak
{}of \markdownRendererInlineMath{( →_β )}.\end{theorem}\markdownRendererInterblockSeparator
{}\begin{proof}Note that it is enough to prove that \markdownRendererInlineMath{( →_β^∗ )}\markdownRendererSoftLineBreak
{}has a diamond property — for each \markdownRendererInlineMath{t_0, t_1, t_2}\markdownRendererSoftLineBreak
{}such that \markdownRendererInlineMath{t_0 →_β^∗ t_1} and \markdownRendererInlineMath{t_0 →_β^∗ t_2}\markdownRendererSoftLineBreak
{}there exists a \markdownRendererInlineMath{t'} such that \markdownRendererInlineMath{t_1 →_β^∗ t'} and\markdownRendererSoftLineBreak
{}\markdownRendererInlineMath{t_2 →_β^∗ t'}:\markdownRendererSoftLineBreak
{}\vspace{-1em}\markdownRendererParagraphSeparator
{}\begin{tikzcd}
  [ampersand replacement = \&]
    t_0 \arrow[r] \arrow[d] \&
    t_1 \arrow[d, dashed]
  \\
    t_2 \arrow[r, dashed] \&
    t'
\end{tikzcd}Then, any ‘mountains’ in the path \markdownRendererInlineMath{t =_β u} can be\markdownRendererSoftLineBreak
{}switched into ‘valleys’.\markdownRendererInterblockSeparator
{}Now, diamond property is proven by the standard\markdownRendererSoftLineBreak
{}technique of parallel reduction \markdownRendererCite{1}+{}{}{martinlof1971theory}\markdownRendererSoftLineBreak
{}which is defined as the smallest relation\markdownRendererSoftLineBreak
{}satisfying:\begin{mathpar}
  \inferrule{ }{t ⇉ t}
  \and
  \inferrule
    {T ⇉ T' \\ U ⇉ U'}
    {Π (x : T).U ⇉ Π (x : T').U'}
  \and
  \inferrule
    {T ⇉ T' \\ t ⇉ t'}
    {λ (x : T).t ⇉ λ (x : T').t'}
  \and
  \inferrule
    {t ⇉ t' \\ u ⇉ u'}
    {t\,u ⇉ t'\,u'}
  \and
  \inferrule
    {t ⇉ t' \\ u ⇉ u'}
    {(λ (x : T).t)\,u ⇉ t'[u'/x]}
\end{mathpar}\markdownRendererParagraphSeparator
{}Intuitively, while \markdownRendererInlineMath{( →_β )} is a single\markdownRendererSoftLineBreak
{}\markdownRendererInlineMath{β}-reduction somewhere in the term and\markdownRendererSoftLineBreak
{}\markdownRendererInlineMath{( →_β )^∗} is a sequence of reductions,\markdownRendererSoftLineBreak
{}\markdownRendererInlineMath{( ⇉ )} is several simultaneous reductions\markdownRendererSoftLineBreak
{}immediately available in the term.\markdownRendererSoftLineBreak
{}Notice, then, that \markdownRendererInlineMath{( →_β ) ⊆ ( ⇉ ) ⊆ ( →_β^∗ )};\markdownRendererSoftLineBreak
{}by monotonicity of reflexive transitive closure,\markdownRendererSoftLineBreak
{}\markdownRendererInlineMath{( →_β^∗ ) ⊆ ( ⇉^∗ ) ⊆ ( →_β^∗ )^∗ = ( →_β^∗ )},\markdownRendererSoftLineBreak
{}so \markdownRendererInlineMath{( ⇉^∗ ) = ( →_β^∗ )}.\markdownRendererSoftLineBreak
{}Thus, it is enough to prove the diamond property of\markdownRendererSoftLineBreak
{}\markdownRendererInlineMath{( ⇉^∗ )}.\markdownRendererSoftLineBreak
{}Diamond property is preserved under reflexive\markdownRendererSoftLineBreak
{}transitive closure, so it is enough to prove\markdownRendererSoftLineBreak
{}diamond property of \markdownRendererInlineMath{( ⇉ )}, which holds by\markdownRendererSoftLineBreak
{}tedious case analysis.\end{proof}\markdownRendererInterblockSeparator
{}Setting aside practical applications of this\markdownRendererSoftLineBreak
{}theorem (to algorithmically check equivalence of\markdownRendererSoftLineBreak
{}terms, we can now search for a common reduction\markdownRendererSoftLineBreak
{}result), it will prove useful already in our proof\markdownRendererSoftLineBreak
{}development.\markdownRendererSoftLineBreak
{}For example, the following “equality inversion”\markdownRendererSoftLineBreak
{}lemmas are direct corollaries:\markdownRendererInterblockSeparator
{}\begin{corollary}[Equality inversion]If \markdownRendererInlineMath{s =_β s'}, then \markdownRendererInlineMath{s = s' ∈ \mathcal{S}};\markdownRendererSoftLineBreak
{}if \markdownRendererInlineMath{Π (x : T).U =_β Π (x : T').U'},\markdownRendererSoftLineBreak
{}then \markdownRendererInlineMath{T =_β T'} and \markdownRendererInlineMath{U =_β U'}.\end{corollary}\markdownRendererInterblockSeparator
{}They will become necessary in proving subject\markdownRendererSoftLineBreak
{}reduction, or type preservation.\markdownRendererSoftLineBreak
{}But first, we will show some structural properties\markdownRendererSoftLineBreak
{}of a typing judgement given in Fig.\markdownRendererSoftLineBreak
{}\markdownRendererLink{\markdownRendererHash{}typing}{\markdownRendererHash{}typing}{#typing}{}.\markdownRendererInterblockSeparator
{}
\markdownRendererSectionEnd \markdownRendererSectionBegin
\markdownRendererHeadingTwo{Properties of a typing judgement}\markdownRendererInterblockSeparator
{}These properties have certainly been proven before\markdownRendererSoftLineBreak
{}for PTS \markdownRendererCite{2}+{}{}{barendregt1992lambda}+{}{}{geuvers1991modular},\markdownRendererSoftLineBreak
{}but our formulation of the typing rules slightly\markdownRendererSoftLineBreak
{}differs in important places, so it is better to\markdownRendererSoftLineBreak
{}recheck that these properties hold.\markdownRendererInterblockSeparator
{}\begin{lemma}[Weakening \proof{rocq:Typing.v\#L50}]The following rule is admissible:\begin{mathpar}
  \dfrac{Γ, Δ ⊢ t : T}
        {Γ, x : U, Δ ⊢ t : T}.
\end{mathpar}\end{lemma}\markdownRendererInterblockSeparator
{}\begin{proof}By structural induction on the typing derivation.\end{proof}\markdownRendererInterblockSeparator
{}\begin{lemma}[Substitution \proof{rocq:Typing.v\#L78}]The following rule is admissible:\begin{mathpar}
  \dfrac{Γ, x : U, Δ ⊢ t : T \quad
         Γ ⊢ u : U}
        {Γ, Δ[u/x] ⊢ t[u/x] : T[u/x]}.
\end{mathpar}\end{lemma}\markdownRendererInterblockSeparator
{}\begin{proof}By structural induction on the typing derivation\markdownRendererSoftLineBreak
{}for \markdownRendererInlineMath{t}.\markdownRendererSoftLineBreak
{}For variables in \markdownRendererInlineMath{Δ}, it is necessary that both\markdownRendererSoftLineBreak
{}\markdownRendererInlineMath{Δ} and \markdownRendererInlineMath{T} get substitutions of \markdownRendererInlineMath{u} for \markdownRendererInlineMath{x} so\markdownRendererSoftLineBreak
{}\textsc{Var} is still applicable;\markdownRendererSoftLineBreak
{}for places in \markdownRendererInlineMath{t} where \markdownRendererInlineMath{x} used to be, we need to\markdownRendererSoftLineBreak
{}apply a weakening lemma to the typing derivation for\markdownRendererSoftLineBreak
{}\markdownRendererInlineMath{u} exactly \markdownRendererInlineMath{|Δ|} times.\markdownRendererSoftLineBreak
{}In the case of \textsc{Conv} it is important that\markdownRendererSoftLineBreak
{}definitional equality of types is preserved by\markdownRendererSoftLineBreak
{}substitution (thanks to characterization of\markdownRendererSoftLineBreak
{}\markdownRendererInlineMath{( =_β )}).\end{proof}\markdownRendererInterblockSeparator
{}In addition, there is a family of so-called\markdownRendererSoftLineBreak
{}\markdownRendererEmphasis{inversion lemmas} \markdownRendererCite{2}+{}{}{barendregt1992lambda}+{}{}{vanbenthemjutting1993typing}\markdownRendererSoftLineBreak
{}stating that a type of well-typed term is largely\markdownRendererSoftLineBreak
{}determined by the shape of a term.\markdownRendererSoftLineBreak
{}They are also proven by straightforward induction\markdownRendererSoftLineBreak
{}on typing derivation\markdownRendererSoftLineBreak
{}(each one is proven separately,\markdownRendererSoftLineBreak
{}no mutual induction involved):\markdownRendererInterblockSeparator
{}\begin{lemma}[Inversion lemmas \proof{rocq:Inversions.v}]\label{inversion}\begin{align}
  Γ ⊢ s : T & ⇒ ∃ s'. (s, s') ∈ \mathsf{Ax} ∧ T =_β s'
  \\
  Γ ⊢ Π (x : T).U : V & ⇒ ∃ s_T, s_U, s.\left\{
  \begin{array}{l}
    (s_T, s_U, s) ∈ \mathsf{Rule} \\
    Γ ⊢ T : s_T \\
    Γ, x : T ⊢ U : s_U \\
    V =_β s
  \end{array}\right. \\
  Γ, x : T, Δ ⊢ x : U & ⇒ T =_β U
  \\
  Γ ⊢ λ (x : T).t : V & ⇒ ∃ s, U.\left\{
  \begin{array}{l}
    Γ ⊢ T : s \\
    Γ, x : T ⊢ t : U \\
    V =_β Π (x : T).U
  \end{array}\right. \\
  Γ ⊢ t\,u : V & ⇒ ∃ U, T.\left\{
  \begin{array}{l}
    Γ ⊢ t : Π (x : U).T \\
    Γ ⊢ u : U \\
    V =_β T[u/x]
  \end{array}\right.
\end{align}\end{lemma}\markdownRendererInterblockSeparator
{}To progress further, we need to make certain\markdownRendererSoftLineBreak
{}assumptions about the PTS in question.\markdownRendererInterblockSeparator
{}\begin{definition}[Functional PTS]A pure type system is called \markdownRendererEmphasis{functional} if\markdownRendererSoftLineBreak
{}\markdownRendererInlineMath{\mathsf{Ax} ⊆ \mathcal{S} × \mathcal{S}} and\markdownRendererSoftLineBreak
{}\markdownRendererInlineMath{\mathsf{Rule} ⊆ \mathcal{S}^2 × \mathcal{S}} are\markdownRendererSoftLineBreak
{}functional relations, that is,\begin{mathpar}
  (s : s') ∧ (s : s”) ⇒ s' = s”
  \quad \text{and} \quad
  (s_1, s_2, s) ∈ \mathsf{Rule} ∧ (s_1, s_2, s') ∈ \mathsf{Rule} ⇒ s = s'
\end{mathpar}\end{definition}\markdownRendererInterblockSeparator
{}\begin{definition}[Total PTS]A pure type system is called \markdownRendererEmphasis{total} if\markdownRendererSoftLineBreak
{}\markdownRendererInlineMath{\mathsf{Ax} ⊆ \mathcal{S} × \mathcal{S}} and\markdownRendererSoftLineBreak
{}\markdownRendererInlineMath{\mathsf{Rule} ⊆ \mathcal{S}^2 × \mathcal{S}} are\markdownRendererSoftLineBreak
{}total relations, that is,\begin{mathpar}
  ∀ s.∃ s'.(s, s') ∈ \mathsf{Ax}
  \quad \text{and} \quad
  ∀ s_1, s_2.∃ s.(s_1, s_2, s) ∈ \mathsf{Rule}
\end{mathpar}\end{definition}\markdownRendererInterblockSeparator
{}\begin{theorem}[Uniqueness of typing \proof{rocq:Inversions.v\#L84}]In a functional PTS, the following rule is\markdownRendererSoftLineBreak
{}admissible:\begin{mathpar}
  \dfrac{Γ ⊢ t : T \quad Γ ⊢ t : T'}
        {T =_β T'}.
\end{mathpar}\end{theorem}\markdownRendererInterblockSeparator
{}\begin{proof}By structural induction on \markdownRendererInlineMath{t} and inversion of\markdownRendererSoftLineBreak
{}the given typing derivations.\markdownRendererInterblockSeparator
{}\markdownRendererUlBegin
\markdownRendererUlItem \markdownRendererInlineMath{t = s ∈ \mathcal{S}}: by functionality of\markdownRendererSoftLineBreak
{}\markdownRendererInlineMath{\mathsf{Ax}}.\markdownRendererUlItemEnd 
\markdownRendererUlItem \markdownRendererInlineMath{t = Π (x : T).U}: inversion of typing gives a\markdownRendererSoftLineBreak
{}pair of (possibly different) sorts for \markdownRendererInlineMath{T} and\markdownRendererSoftLineBreak
{}another pair for \markdownRendererInlineMath{U}.\markdownRendererSoftLineBreak
{}By induction we get that ranks inside pairs are\markdownRendererSoftLineBreak
{}definitionally equal, which by equality\markdownRendererSoftLineBreak
{}inversion gives strictly equal sorts.\markdownRendererSoftLineBreak
{}Then what is needed holds by the functionality of\markdownRendererSoftLineBreak
{}\markdownRendererInlineMath{\mathsf{Rule}}.\markdownRendererUlItemEnd 
\markdownRendererUlItem The case for \markdownRendererInlineMath{t = x ∈ \mathsf{Var}} is trivial.\markdownRendererUlItemEnd 
\markdownRendererUlItem \markdownRendererInlineMath{t = (λ (x : T).u)}: by induction hypotheses and\markdownRendererSoftLineBreak
{}by congruence of \markdownRendererInlineMath{( =_β )}.\markdownRendererUlItemEnd 
\markdownRendererUlItem \markdownRendererInlineMath{t = u\,v}: by typing inversion and by induction\markdownRendererSoftLineBreak
{}hypothesis for \markdownRendererInlineMath{u} we get \markdownRendererInlineMath{u : Π (x : T_1).U_1}\markdownRendererSoftLineBreak
{}and \markdownRendererInlineMath{Π (x : T_1).U_1 =_β Π (x : T_2).U_2}.\markdownRendererSoftLineBreak
{}We need \markdownRendererInlineMath{U_1[v/x] =_β U_2[v/x]} which we get\markdownRendererSoftLineBreak
{}from equality inversion and preservation of\markdownRendererSoftLineBreak
{}def.eq.\markdownRendererSoftLineBreak
{}under substitution. \vspace{-1em}\markdownRendererUlItemEnd 
\markdownRendererUlEnd \end{proof}\markdownRendererInterblockSeparator
{}\vspace{-.5em}\markdownRendererSoftLineBreak
{}In our proof of normalization, we are going to\markdownRendererSoftLineBreak
{}rely on yet another property: if there is a\markdownRendererSoftLineBreak
{}derivation for \markdownRendererInlineMath{t : T}, then \markdownRendererInlineMath{T} is itself\markdownRendererSoftLineBreak
{}typeable with some sort.\markdownRendererSoftLineBreak
{}However, it does not hold in our current\markdownRendererSoftLineBreak
{}formulation of typing rules: in \markdownRendererInlineMath{Γ ⊢ x : T},\markdownRendererSoftLineBreak
{}nothing prevents a context \markdownRendererInlineMath{Γ} from containing\markdownRendererSoftLineBreak
{}an entry\markdownRendererSoftLineBreak
{}\markdownRendererInlineMath{(x : T')} where \markdownRendererInlineMath{T'} is not typeable.\markdownRendererSoftLineBreak
{}Thus, we need an additional assumption.\markdownRendererInterblockSeparator
{}\begin{definition}[Context well-formedness]Context \markdownRendererInlineMath{Γ} is \markdownRendererEmphasis{well-formed} (written \markdownRendererInlineMath{⊢ Γ}) if\markdownRendererSoftLineBreak
{}every type in it is typed by some sort.\markdownRendererSoftLineBreak
{}Formally, a set of well-formed contexts is the least\markdownRendererSoftLineBreak
{}set closed by the following rules:\begin{mathpar}
  \dfrac{}{⊢ ⋅} \quad \text{and} \quad
  \dfrac{⊢ Γ \quad Γ ⊢ T : s}
        {⊢ Γ, x : T}
\end{mathpar}\markdownRendererInterblockSeparator
{}(where \markdownRendererInlineMath{⋅} denotes an empty context).\end{definition}\markdownRendererInterblockSeparator
{}Usually, when presenting a dependently typed\markdownRendererSoftLineBreak
{}system, context well-formedness is added as a\markdownRendererSoftLineBreak
{}premise to typing rules which would otherwise be\markdownRendererSoftLineBreak
{}premise-free (such as \textsc{Ax} and \textsc{Var}\markdownRendererSoftLineBreak
{}in a PTS).\markdownRendererSoftLineBreak
{}Here we instead provide it as an additional\markdownRendererSoftLineBreak
{}assumption only when necessary;\markdownRendererSoftLineBreak
{}in addition to leading to simpler metatheory (we\markdownRendererSoftLineBreak
{}avoid an otherwise mutually inductive definition\markdownRendererSoftLineBreak
{}of \markdownRendererInlineMath{( ⊢ \_)} and \markdownRendererInlineMath{(\_ ⊢ \_ : \_)}),\markdownRendererSoftLineBreak
{}this is closer to how decision procedures for\markdownRendererSoftLineBreak
{}typing judgements are implemented.\markdownRendererInterblockSeparator
{}\begin{theorem}[Typing yields correct type \proof{rocq:Preservation.v\#L65}]\label{type-is-correct}In a total PTS, if \markdownRendererInlineMath{Γ ⊢ t : T} and \markdownRendererInlineMath{⊢ Γ},\markdownRendererSoftLineBreak
{}then there exists \markdownRendererInlineMath{s ∈ \mathcal{S}} such that\markdownRendererSoftLineBreak
{}\markdownRendererInlineMath{Γ ⊢ T : s}.\end{theorem}\markdownRendererInterblockSeparator
{}\begin{proof}By structural induction on the typing derivation.\markdownRendererInterblockSeparator
{}\markdownRendererUlBegin
\markdownRendererUlItem \textsc{Ax}, \textsc{Rule}: by totality of\markdownRendererSoftLineBreak
{}\markdownRendererInlineMath{\mathsf{Ax}}.\markdownRendererUlItemEnd 
\markdownRendererUlItem \textsc{Var}: follows from \markdownRendererInlineMath{⊢ Γ}.\markdownRendererUlItemEnd 
\markdownRendererUlItem \textsc{Abs}: by totality of \markdownRendererInlineMath{\mathsf{Rule}}.\markdownRendererUlItemEnd 
\markdownRendererUlItem \textsc{App}: \markdownRendererInlineMath{Γ ⊢ t\,u : T[u/x]}.\markdownRendererSoftLineBreak
{}From induction hypothesis for \markdownRendererInlineMath{t} we know that\markdownRendererSoftLineBreak
{}\markdownRendererInlineMath{Γ ⊢ t : Π (x : U).T : s} for some \markdownRendererInlineMath{s},\markdownRendererSoftLineBreak
{}then by typing inversion for \markdownRendererInlineMath{Π (x : U).T} we\markdownRendererSoftLineBreak
{}get that \markdownRendererInlineMath{T} is typable with some sort.\markdownRendererSoftLineBreak
{}Then what is needed follows from a substitution\markdownRendererSoftLineBreak
{}lemma.\markdownRendererUlItemEnd 
\markdownRendererUlItem \textsc{Conv}: what is needed is present in the\markdownRendererSoftLineBreak
{}premise. \vspace{-1em}\markdownRendererUlItemEnd 
\markdownRendererUlEnd \end{proof}\markdownRendererInterblockSeparator
{}\vspace{-.5em}\markdownRendererSoftLineBreak
{}Totality is not a strictly necessary requirement —\markdownRendererSoftLineBreak
{}similar true propositions can be stated for\markdownRendererSoftLineBreak
{}non-total PTS, for example, in CoC typing yields\markdownRendererSoftLineBreak
{}type which is either itself typable or equal to\markdownRendererSoftLineBreak
{}\markdownRendererInlineMath{□} \markdownRendererCite{2}+{}{}{barendregt1992lambda}+{}{}{luo1994computation}.\markdownRendererSoftLineBreak
{}But \PTSrk{} is indeed total so the theorem above\markdownRendererSoftLineBreak
{}is enough for our purposes.\markdownRendererInterblockSeparator
{}\begin{theorem}[Subject reduction \proof{rocq:Preservation.v\#L92}]In a total PTS, the following rule is admissible:\begin{mathpar}
  \dfrac{Γ ⊢ t : T \quad ⊢ Γ
         \quad t →_β t'}
        {Γ ⊢ t' : T}
\end{mathpar}\end{theorem}\markdownRendererInterblockSeparator
{}\begin{proof}By structural induction over derivation of\markdownRendererSoftLineBreak
{}\markdownRendererInlineMath{t →_β t'}.\markdownRendererSoftLineBreak
{}Inversion of a typing judgement yields all\markdownRendererSoftLineBreak
{}necessary info, albeit because of it types are now\markdownRendererSoftLineBreak
{}given only up to definitional equality, so we have\markdownRendererSoftLineBreak
{}to apply the conversion rule in each case.\markdownRendererSoftLineBreak
{}Crucially, it requires the type we convert into to\markdownRendererSoftLineBreak
{}be typable by a sort, so that is where we use\markdownRendererSoftLineBreak
{}Theorem \markdownRendererLink{\markdownRendererHash{}type-is-correct}{\markdownRendererHash{}type-is-correct}{#type-is-correct}{}.\markdownRendererInterblockSeparator
{}The only tricky cases are reductions of a domain\markdownRendererSoftLineBreak
{}type in \markdownRendererInlineMath{Π (x : T).U} and \markdownRendererInlineMath{λ (x : T).t}.\markdownRendererSoftLineBreak
{}Let us consider the case of the following\markdownRendererSoftLineBreak
{}reduction:\begin{mathpar}
  λ (x : T).t →_β λ (x : T').t.
\end{mathpar}\markdownRendererInterblockSeparator
{}To complete this case, we have to prove\markdownRendererSoftLineBreak
{}that \markdownRendererInlineMath{Γ, x : T' ⊢ t : U} given that\markdownRendererSoftLineBreak
{}\markdownRendererInlineMath{Γ, x : T ⊢ t : U}.\markdownRendererSoftLineBreak
{}By appropriately strengthening the statement,\markdownRendererSoftLineBreak
{}we get that it is enough to prove admissibility of\markdownRendererSoftLineBreak
{}the following rule called \markdownRendererEmphasis{context reduction}:\begin{mathpar}
  \dfrac{Γ, x : U, Δ ⊢ t : T \quad
         ⊢ Γ, x : U, Δ \quad
         U →_β U'}
        {Γ, x : U', Δ ⊢ t : T}
\end{mathpar}\markdownRendererParagraphSeparator
{}The proof is, again, by induction on the typing\markdownRendererSoftLineBreak
{}derivation.\markdownRendererSoftLineBreak
{}The only non-trivial case is with \textsc{Var}\markdownRendererSoftLineBreak
{}rule used on \markdownRendererInlineMath{x}; it is taken care of by\markdownRendererSoftLineBreak
{}the appropriate use of \textsc{Conv} rule to convert\markdownRendererSoftLineBreak
{}\markdownRendererInlineMath{x : U'} to \markdownRendererInlineMath{x : U}.\markdownRendererSoftLineBreak
{}To use conversion, we have to prove that \markdownRendererInlineMath{U} is\markdownRendererSoftLineBreak
{}well-typed; thankfully, it follows from\markdownRendererSoftLineBreak
{}\markdownRendererInlineMath{⊢ Γ, x : U, Δ}.\markdownRendererParagraphSeparator
{}Note that context reduction is admissible in\markdownRendererSoftLineBreak
{}arbitrary PTS!\markdownRendererSoftLineBreak
{}We do not need it anywhere else, though,\markdownRendererSoftLineBreak
{}so we did not state it as a separate lemma.\end{proof}\markdownRendererInterblockSeparator
{}Characterization of definitional equality,\markdownRendererSoftLineBreak
{}uniqueness of typing and subject reduction\markdownRendererSoftLineBreak
{}together imply the following corollary we are\markdownRendererSoftLineBreak
{}going to use a few times in our normalization\markdownRendererSoftLineBreak
{}proof:\markdownRendererInterblockSeparator
{}\begin{corollary}[Type stability]In a total functional PTS, the following holds:\begin{mathpar}
  \left\{
  \begin{array}{l}
    Γ ⊢ t : T\\
    Γ ⊢ t' : T'\\ ⊢ Γ \\
    t =_β t'
  \end{array}\right. \quad ⇒ \quad
      T =_β T'.
\end{mathpar}\end{corollary}\markdownRendererInterblockSeparator
{}
\markdownRendererSectionEnd \markdownRendererSectionBegin
\markdownRendererHeadingTwo{Properties of strongly normalizing terms}\markdownRendererInterblockSeparator
{}We are almost done with classical results!\markdownRendererSoftLineBreak
{}What is left now is to state which terms are\markdownRendererSoftLineBreak
{}strongly normalizing and collect properties of\markdownRendererSoftLineBreak
{}s.n.\markdownRendererSoftLineBreak
{}terms which allow our soundness proof to go\markdownRendererSoftLineBreak
{}through.\markdownRendererParagraphSeparator
{}Informally, a term is strongly normalizing iff\markdownRendererSoftLineBreak
{}there are no infinite reduction chains starting\markdownRendererSoftLineBreak
{}from it.\markdownRendererSoftLineBreak
{}In constructive setting and in formalization,\markdownRendererSoftLineBreak
{}the following definition is a bit better \markdownRendererCite{2}+{}{}{altenkirch1993formalization}+{}{}{vanraamsdonk1999perpetual}:\markdownRendererInterblockSeparator
{}\begin{definition}[Strongly normalizing term]\label{sn-definition}Set of s.n.\markdownRendererSoftLineBreak
{}terms is the smallest set \markdownRendererInlineMath{\mathsf{sn}} in which\markdownRendererSoftLineBreak
{}the following rule is valid:\begin{mathpar}
  \dfrac{∀ t', (t →_β t') ⇒ t' ∈ \mathsf{sn}}{t ∈ \mathsf{sn}}.
\end{mathpar}\markdownRendererInterblockSeparator
{}In other words, a term is strongly normalizing iff\markdownRendererSoftLineBreak
{}all possible reductions from it are strongly\markdownRendererSoftLineBreak
{}normalizing.\end{definition}\markdownRendererInterblockSeparator
{}Thus, terms which cannot be reduced further (also\markdownRendererSoftLineBreak
{}known as terms \markdownRendererEmphasis{in normal form}) are strongly\markdownRendererSoftLineBreak
{}normalizing; then, terms which directly reduce\markdownRendererSoftLineBreak
{}into normal forms are also strongly normalizing;\markdownRendererSoftLineBreak
{}and so on.\markdownRendererSoftLineBreak
{}Therefore, by this definition,\markdownRendererSoftLineBreak
{}strongly normalizing terms are those that every\markdownRendererSoftLineBreak
{}reduction chain starting from them is finite.\markdownRendererParagraphSeparator
{}The following properties are then provable by\markdownRendererSoftLineBreak
{}definition of \markdownRendererInlineMath{\mathsf{sn}} and by structural\markdownRendererSoftLineBreak
{}induction on available derivations of\markdownRendererSoftLineBreak
{}\markdownRendererInlineMath{\mathsf{sn}} for terms in question:\markdownRendererInterblockSeparator
{}\begin{lemma}[Structural properties of $\mathsf{sn}$ \proof{rocq:SN.v}]\label{sn-props}\begin{mathpar}
  \inferrule{ }{x, s ∈ \mathsf{sn}}\and
  \inferrule
    {T ∈ \mathsf{sn} \\ U ∈ \mathsf{sn}}
    {Π (x : T).U ∈ \mathsf{sn}}\and
  \inferrule
    {T ∈ \mathsf{sn} \\ t ∈ \mathsf{sn}}
    {λ (x : T).t ∈ \mathsf{sn}}\and
  \inferrule{t\,u ∈ \mathsf{sn}}{t ∈ \mathsf{sn}}\and
  \inferrule{t[u/x] ∈ \mathsf{sn}}{t ∈ \mathsf{sn}}
\end{mathpar}\end{lemma}\markdownRendererInterblockSeparator
{}Application, of course, is the only nontrivial\markdownRendererSoftLineBreak
{}case.\markdownRendererSoftLineBreak
{}In the case of \markdownRendererInlineMath{β}-reduction, strong normalization of\markdownRendererSoftLineBreak
{}both \markdownRendererInlineMath{(λ (x : T).t)} and \markdownRendererInlineMath{u} does not imply\markdownRendererSoftLineBreak
{}anything about normalization behaviour of\markdownRendererSoftLineBreak
{}\markdownRendererInlineMath{t[u/x]}: consider a classical example from\markdownRendererSoftLineBreak
{}untyped lambda calculus \markdownRendererCite{1}+{}{}{barendregt1984lambda},\markdownRendererSoftLineBreak
{}\markdownRendererInlineMath{ω = (λ (x : T).x\,x)} and \markdownRendererInlineMath{Ω = ω \,ω};\markdownRendererSoftLineBreak
{}while not typable, \markdownRendererInlineMath{ω} is strongly\markdownRendererSoftLineBreak
{}normalizing, while \markdownRendererInlineMath{Ω} decidedly is not.\markdownRendererParagraphSeparator
{}The situation changes when the head of the application is\markdownRendererSoftLineBreak
{}not an abstraction at all.\markdownRendererSoftLineBreak
{}This leads us to the definition of an \markdownRendererEmphasis{atomic\markdownRendererSoftLineBreak
{}term}.\markdownRendererInterblockSeparator
{}\begin{definition}[Atomic term]The set of atomic terms is the smallest set\markdownRendererSoftLineBreak
{}that satisfies the following conditions:\markdownRendererInterblockSeparator
{}\markdownRendererOlBegin
\markdownRendererOlItemWithNumber{1}\markdownRendererInlineMath{s}, \markdownRendererInlineMath{Π (x : S).T} and \markdownRendererInlineMath{y} are atomic for all\markdownRendererSoftLineBreak
{}variables \markdownRendererInlineMath{x},\markdownRendererInlineMath{y}, sort \markdownRendererInlineMath{s} and types \markdownRendererInlineMath{S},\markdownRendererInlineMath{T};\markdownRendererOlItemEnd 
\markdownRendererOlItemWithNumber{2}If \markdownRendererInlineMath{t} is atomic, then \markdownRendererInlineMath{t\,u} is also atomic.\markdownRendererOlItemEnd 
\markdownRendererOlEnd \end{definition}\markdownRendererInterblockSeparator
{}Atomic terms are useful precisely because their\markdownRendererSoftLineBreak
{}strong normalization survives application:\markdownRendererInterblockSeparator
{}\begin{lemma}\label{sn-atomic}If \markdownRendererInlineMath{t, u ∈ \mathsf{sn}} and \markdownRendererInlineMath{t} is atomic,\markdownRendererSoftLineBreak
{}then \markdownRendererInlineMath{t\,u ∈ \mathsf{sn}}.\;\proof{rocq:SN.v\#L78}\end{lemma}\markdownRendererInterblockSeparator
{}\begin{proof}By induction on derivations of \markdownRendererInlineMath{\mathsf{sn}} for\markdownRendererSoftLineBreak
{}\markdownRendererInlineMath{t} and \markdownRendererInlineMath{u}.\markdownRendererSoftLineBreak
{}Consider possible redexes in \markdownRendererInlineMath{t\,u}:\markdownRendererSoftLineBreak
{}\markdownRendererInlineMath{t\,u} is not itself a redex since \markdownRendererInlineMath{t} is atomic\markdownRendererSoftLineBreak
{}and hence not a lambda; the only possible redexes\markdownRendererSoftLineBreak
{}are either inside \markdownRendererInlineMath{t} or inside \markdownRendererInlineMath{u}.\markdownRendererSoftLineBreak
{}Applying corresponding induction hypotheses,\markdownRendererSoftLineBreak
{}we reach the desired result.\markdownRendererSoftLineBreak
{}The only caveat is that, for induction to go\markdownRendererSoftLineBreak
{}through, we need to check that atomicity of \markdownRendererInlineMath{t} is\markdownRendererSoftLineBreak
{}preserved by reduction; this is indeed true by\markdownRendererSoftLineBreak
{}induction on derivation of atomicity for \markdownRendererInlineMath{t}.\end{proof}\markdownRendererInterblockSeparator
{}A dual view on strong normalization of application\markdownRendererSoftLineBreak
{}is the following: if we already know that\markdownRendererSoftLineBreak
{}\markdownRendererInlineMath{t[u/x] ∈ \mathsf{sn}}, then the most tricky case\markdownRendererSoftLineBreak
{}is already taken care of.\markdownRendererInterblockSeparator
{}\begin{lemma}[Weak head expansion \cite{joachimski2003short, vanraamsdonk1995normalisation, altenkirch1993constructions}\;\proof{rocq:SN.v\#L167}]\label{sn-whe}\begin{mathpar}
  \left\{
  \begin{array}{l}
    t[u/x]u_1… u_n ∈ \mathsf{sn} \\
    U, u ∈ \mathsf{sn}
  \end{array}\right. ⇒ (λ (x : U).t)u u_1… u_n ∈ \mathsf{sn}.
\end{mathpar}\end{lemma}\markdownRendererInterblockSeparator
{}\begin{proof}By induction on \markdownRendererInlineMath{n}.\markdownRendererInterblockSeparator
{}\markdownRendererUlBegin
\markdownRendererUlItem \markdownRendererInlineMath{n = 0}.\markdownRendererSoftLineBreak
{}We need to prove that if\markdownRendererSoftLineBreak
{}\markdownRendererInlineMath{t[u/x], u, U ∈ \mathsf{sn}},\markdownRendererSoftLineBreak
{}then \markdownRendererInlineMath{(λ (x : U).t)u ∈ \mathsf{sn}}.\markdownRendererSoftLineBreak
{}Note that \markdownRendererInlineMath{t[u/x] ∈ \mathsf{sn}} implies\markdownRendererSoftLineBreak
{}\markdownRendererInlineMath{t ∈ \mathsf{sn}}; then, by structural induction\markdownRendererSoftLineBreak
{}on derivations of \markdownRendererInlineMath{\mathsf{sn}} for \markdownRendererInlineMath{t, u, U} we\markdownRendererSoftLineBreak
{}get what we want by considering all possible\markdownRendererSoftLineBreak
{}reductions.\markdownRendererUlItemEnd 
\markdownRendererUlItem \markdownRendererInlineMath{n → (n + 1)}.\markdownRendererSoftLineBreak
{}By the induction hypothesis, it is sufficient to show\markdownRendererSoftLineBreak
{}the following:\markdownRendererParagraphSeparator
{}\markdownRendererDisplayMath{\left\{
\begin{array}{l}
(λ (x : U).t)u u_1… u_n ∈ \mathsf{sn} \\
t[u/x]u_1… u_{n + 1} ∈ \mathsf{sn}
\end{array}\right. ⇒ (λ (x : U).t)u u_1… u_n u_{n + 1} ∈ \mathsf{sn}.
}\markdownRendererSoftLineBreak
{}This is true by structural induction on\markdownRendererSoftLineBreak
{}available evidence.\vspace{-1em}\markdownRendererUlItemEnd 
\markdownRendererUlEnd \end{proof}\markdownRendererInterblockSeparator
{}
\markdownRendererSectionEnd 
\markdownRendererSectionEnd \markdownRendererSectionBegin
\markdownRendererHeaderAttributeContextBegin
\markdownRendererAttributeIdentifier{rk-sn}\markdownRendererHeadingOne{Strong normalization of \PTSrk{}}\markdownRendererHeaderAttributeContextEnd{}\markdownRendererInterblockSeparator
{}\begin{theorem}[Strong normalization of \PTSrk{}\;\proof{rocq:Ranks.v\#L217}]\label{normalization}\begin{mathpar}
  Γ ⊢ t : T\, ∧ ⊢ Γ ⇒ t ∈ \mathsf{sn}.
\end{mathpar}\end{theorem}\markdownRendererInterblockSeparator
{}\begin{proof}By induction on the typing derivation.\markdownRendererSoftLineBreak
{}Case \textsc{Conv} is trivial; due to Lemma \markdownRendererLink{\markdownRendererHash{}sn-props}{\markdownRendererHash{}sn-props}{#sn-props}{},\markdownRendererSoftLineBreak
{}cases \textsc{Ax}, \textsc{Rule},\markdownRendererSoftLineBreak
{}\textsc{Var} and \textsc{Abs} present no problem\markdownRendererSoftLineBreak
{}as well.\markdownRendererSoftLineBreak
{}The only problematic case is application:\begin{mathpar}
  \left\{
  \begin{array}{l}
    Γ ⊢ u : Π (x : T).U \\
    Γ ⊢ t : T \\ ⊢ Γ \\
    t, u ∈ \mathsf{sn}
  \end{array}\right. ⇒ u\,t ∈ \mathsf{sn}
\end{mathpar}\markdownRendererInterblockSeparator
{}Now, here is the part where stratification\markdownRendererSoftLineBreak
{}according to ranks comes in useful.\markdownRendererSoftLineBreak
{}As typing yields the correct type, there is some \markdownRendererInlineMath{r}\markdownRendererSoftLineBreak
{}such that \markdownRendererInlineMath{Γ ⊢ Π (x : T).U : \rk{r}}.\markdownRendererSoftLineBreak
{}Let us explicitly add this to the list of\markdownRendererSoftLineBreak
{}premises; we will call the following proposition\markdownRendererSoftLineBreak
{}\markdownRendererEmphasis{Application Lemma}:\begin{mathpar}
  \mathrm{AL}_r ≜ \quad
         \text{“}\left\{
  \begin{array}{l}
    Γ ⊢ Π (x : T).U : \rk{r}\\
    Γ ⊢ u : Π (x : T).U \\
    Γ ⊢ t : T \\ ⊢ Γ \\
    t, u ∈ \mathsf{sn}
  \end{array}\right. ⇒ u\,t ∈ \mathsf{sn}\text{”}
\end{mathpar}\markdownRendererInterblockSeparator
{}What indexing by \markdownRendererInlineMath{r} bought us is that now we can\markdownRendererSoftLineBreak
{}try proving this lemma by induction on \markdownRendererInlineMath{r}.\markdownRendererSoftLineBreak
{}Again, as it always is with strong normalization\markdownRendererSoftLineBreak
{}of application, the problem is in case where\markdownRendererSoftLineBreak
{}the application head is a lambda and a \markdownRendererInlineMath{β}-reduction\markdownRendererSoftLineBreak
{}takes place.\markdownRendererSoftLineBreak
{}This calls for another lemma,\markdownRendererSoftLineBreak
{}\markdownRendererEmphasis{Substitution Lemma}:\begin{mathpar}
  \mathrm{SL}_r ≜ \quad
         \text{“}\left\{
  \begin{array}{l}
    Γ, x : T, Δ ⊢ u : U\\
    Γ ⊢ T : \rk{r}\\
    Γ ⊢ t : T\\ ⊢ Γ \\
    t, u ∈ \mathsf{sn}
  \end{array}\right. ⇒ u[t/x] ∈ \mathsf{sn}\text{”}
\end{mathpar}\markdownRendererInterblockSeparator
{}A proof of \markdownRendererInlineMath{\mathrm{AL}_r} for all \markdownRendererInlineMath{r} (which is\markdownRendererSoftLineBreak
{}what we are after) is achieved through complete\markdownRendererSoftLineBreak
{}mutual induction with \markdownRendererInlineMath{\mathrm{SL}_r}.\markdownRendererSoftLineBreak
{}Thus, there are two main parts:\markdownRendererInterblockSeparator
{}\markdownRendererUlBegin
\markdownRendererUlItem \markdownRendererInlineMath{(∀ s < r.\;\mathrm{SL}_s) ⇒ \mathrm{AL}_r}:\markdownRendererSoftLineBreak
{}by induction on derivations of\markdownRendererSoftLineBreak
{}\markdownRendererInlineMath{u, t ∈ \mathsf{sn}}.\markdownRendererSoftLineBreak
{}Consider possible reductions of \markdownRendererInlineMath{u\,t}:\markdownRendererSoftLineBreak
{}most of them happen either in \markdownRendererInlineMath{u} or in \markdownRendererInlineMath{t},\markdownRendererSoftLineBreak
{}so that results of reductions from \markdownRendererInlineMath{u\,t} are\markdownRendererSoftLineBreak
{}strongly normalizing by inductive hypotheses.\markdownRendererSoftLineBreak
{}The only problematic case is when\markdownRendererSoftLineBreak
{}\markdownRendererInlineMath{u = λ (x : T').v} and the reduction is\markdownRendererSoftLineBreak
{}\markdownRendererInlineMath{β}-reduction.\markdownRendererSoftLineBreak
{}By inversion of typing for \markdownRendererInlineMath{u}, it should be the\markdownRendererSoftLineBreak
{}case that \markdownRendererInlineMath{T' =_β T}; and by inversion of typing\markdownRendererSoftLineBreak
{}for \markdownRendererInlineMath{Π (x : T).U}, we get that \markdownRendererInlineMath{T : \rk{s}} and\markdownRendererSoftLineBreak
{}\markdownRendererInlineMath{U : \rk{t}} such that \markdownRendererInlineMath{r = \max(s + 1, t)}.\markdownRendererSoftLineBreak
{}Then, by type stability, \markdownRendererInlineMath{T' : \rk{s}},\markdownRendererSoftLineBreak
{}and furthermore, \markdownRendererInlineMath{s < r}, so we can apply\markdownRendererSoftLineBreak
{}a substitution lemma.\markdownRendererUlItemEnd 
\markdownRendererUlItem \markdownRendererInlineMath{(∀ s ≤ r.\;\mathrm{AL}_s) ⇒ \mathrm{SL}_r}.\markdownRendererSoftLineBreak
{}So we need to prove \markdownRendererInlineMath{u[t/x] ∈ \mathsf{sn}}.\markdownRendererSoftLineBreak
{}We do it by simultaneous induction on derivation\markdownRendererSoftLineBreak
{}of \markdownRendererInlineMath{u ∈ \mathsf{sn}} and structural induction on\markdownRendererSoftLineBreak
{}\markdownRendererInlineMath{u}.\markdownRendererSoftLineBreak
{}As substitution propagates through \markdownRendererInlineMath{u},\markdownRendererSoftLineBreak
{}most cases are satisfied by Lemma \markdownRendererLink{\markdownRendererHash{}sn-props}{\markdownRendererHash{}sn-props}{#sn-props}{}.\markdownRendererSoftLineBreak
{}A funny case is when \markdownRendererInlineMath{u = x}: then,\markdownRendererSoftLineBreak
{}a premise that \markdownRendererInlineMath{t ∈ \mathsf{sn}} helps.\markdownRendererInterblockSeparator
{}The only problematic case is\markdownRendererSoftLineBreak
{}application \markdownRendererInlineMath{u = u_1\,u_2}; we consider three\markdownRendererSoftLineBreak
{}cases here:\markdownRendererInterblockSeparator
{}\markdownRendererUlBegin
\markdownRendererUlItem \markdownRendererInlineMath{u_1[t/x]} is atomic — handled by Lemma \markdownRendererLink{\markdownRendererHash{}sn-props}{\markdownRendererHash{}sn-props}{#sn-props}{}.\markdownRendererUlItemEnd 
\markdownRendererUlItem \markdownRendererInlineMath{u_1} is not atomic — by virtue of it being\markdownRendererSoftLineBreak
{}well-typed, it means that\markdownRendererSoftLineBreak
{}\markdownRendererInlineMath{u = (λ (x : V).u')v_1 … v_k u_2}.\markdownRendererSoftLineBreak
{}All possible reductions from \markdownRendererInlineMath{u[t/x]} are then\markdownRendererSoftLineBreak
{}handled by induction hypotheses.\markdownRendererUlItemEnd 
\markdownRendererUlItem \markdownRendererInlineMath{u_1} is atomic, but \markdownRendererInlineMath{u_1[t/x]} is not — by\markdownRendererSoftLineBreak
{}virtue of \markdownRendererInlineMath{u_1} being well-typed,\markdownRendererSoftLineBreak
{}the only case when this is possible is when\markdownRendererSoftLineBreak
{}\markdownRendererInlineMath{u_1 = x v_1… v_k}, thus\markdownRendererSoftLineBreak
{}\markdownRendererInlineMath{u[t/x] = t (v_1[t/x])… (v_k[t/x])(u_2[t/x])}.\markdownRendererSoftLineBreak
{}From induction hypotheses we know that all\markdownRendererSoftLineBreak
{}\markdownRendererInlineMath{v_i[t/x]} and \markdownRendererInlineMath{u_2[t/x]} are s.n.;\markdownRendererSoftLineBreak
{}\markdownRendererInlineMath{t} is s.n.\markdownRendererSoftLineBreak
{}by premise and its rank is \markdownRendererInlineMath{r}.\markdownRendererSoftLineBreak
{}Then \markdownRendererInlineMath{t v_1[t/x]} is s.n.\markdownRendererSoftLineBreak
{}by application lemma for \markdownRendererInlineMath{r};\markdownRendererSoftLineBreak
{}but note that rank of \markdownRendererInlineMath{t v_1[t/x]} is not\markdownRendererSoftLineBreak
{}greater than \markdownRendererInlineMath{r} so application lemma applies\markdownRendererSoftLineBreak
{}again.\markdownRendererSoftLineBreak
{}Thus, by induction on \markdownRendererInlineMath{k}, the whole chain of\markdownRendererSoftLineBreak
{}applications is s.n.\vspace{-1em}\markdownRendererUlItemEnd 
\markdownRendererUlEnd \markdownRendererUlItemEnd 
\markdownRendererUlEnd \end{proof}\markdownRendererInterblockSeparator
{}\vspace{-.5em}\markdownRendererSoftLineBreak
{}This is a retelling of a proof that our friend\markdownRendererSoftLineBreak
{}Aleksandr Levin got after half-jokingly entering\markdownRendererSoftLineBreak
{}the following prompt into GPT-6 Pro:\markdownRendererInterblockSeparator
{}\markdownRendererBlockQuoteBegin
Consider the following pure type system:\markdownRendererInterblockSeparator
{}\markdownRendererInputFencedCode{./_markdown_PTS-rk/e312f8a03844e0000f76be092ef72e2d.verbatim}{text}{text}\markdownRendererBlockQuoteEnd \markdownRendererInterblockSeparator
{}\markdownRendererBlockQuoteBegin
Is it strongly normalizing or not?\markdownRendererBlockQuoteEnd \markdownRendererInterblockSeparator
{}After approximately 20 minutes,\markdownRendererSoftLineBreak
{}LLM outputted the proof which we retold above.\markdownRendererSoftLineBreak
{}Looking at it in retrospect, it seems as if we\markdownRendererSoftLineBreak
{}stratified functions by ranks with intent for such\markdownRendererSoftLineBreak
{}a direct proof to go through; honestly,\markdownRendererSoftLineBreak
{}we did not think of that and were eagerly trying\markdownRendererSoftLineBreak
{}to adapt well-known approaches for proving\markdownRendererSoftLineBreak
{}normalization of pure type systems employing\markdownRendererSoftLineBreak
{}either saturated sets or candidates of\markdownRendererSoftLineBreak
{}reducibility \markdownRendererCite{4}+{}{}{tait1967intensional}+{}{}{girard1972interpretation}+{}{}{girard1989proofs}+{}{}{geuvers1991modular},\markdownRendererSoftLineBreak
{}which did not work because our “hierarchy” of\markdownRendererSoftLineBreak
{}ranks is not well-founded enough because of axioms\markdownRendererSoftLineBreak
{}\markdownRendererInlineMath{\rk{n} : \rk{0}}.\markdownRendererSoftLineBreak
{}This proof, however, does not even bother with it,\markdownRendererSoftLineBreak
{}which means that we could pick \markdownRendererEmphasis{any} relation as\markdownRendererSoftLineBreak
{}\markdownRendererInlineMath{\mathsf{Ax}}.\markdownRendererSoftLineBreak
{}Nevertheless, putting all sorts into \markdownRendererInlineMath{\rk{0}}\markdownRendererSoftLineBreak
{}seems the most sensible approach to us since,\markdownRendererSoftLineBreak
{}informally, sort is kind of a constant which can\markdownRendererSoftLineBreak
{}not be converted into a function.\markdownRendererInterblockSeparator
{}\markdownRendererSectionBegin
\markdownRendererHeaderAttributeContextBegin
\markdownRendererAttributeIdentifier{formalization}\markdownRendererHeadingTwo{Formalizing normalization of \PTSrk{} in Rocq}\markdownRendererHeaderAttributeContextEnd{}\markdownRendererInterblockSeparator
{}Given the high technicality of the proof for\markdownRendererSoftLineBreak
{}strong normalization which involves wrangling of\markdownRendererSoftLineBreak
{}unusual definition for strong normalization and\markdownRendererSoftLineBreak
{}complex induction patterns (and given\markdownRendererSoftLineBreak
{}that proof was supplied by the probabilistic\markdownRendererSoftLineBreak
{}machine), we decided to formalize it so that a\markdownRendererSoftLineBreak
{}sound simple algorithm would check it for us.\markdownRendererSoftLineBreak
{}One of the natural choices for formalization of\markdownRendererSoftLineBreak
{}programming languages is Rocq \markdownRendererCite{2}+{}{}{rocq2025manual}+{}{}{bertot2004interactive},\markdownRendererSoftLineBreak
{}an interactive theorem prover based on\markdownRendererSoftLineBreak
{}Polymorphic, Cumulative Calculus of Inductive\markdownRendererSoftLineBreak
{}Constructions \markdownRendererCite{2}+{}{}{sozeau2014universe}+{}{}{timany2018cumulative}.\markdownRendererSoftLineBreak
{}As is the case with similar typed calculi,\markdownRendererSoftLineBreak
{}it is both a programming language and a\markdownRendererSoftLineBreak
{}higher-order logic, and correctness of a proof is\markdownRendererSoftLineBreak
{}checked by the type checker; in addition to that,\markdownRendererSoftLineBreak
{}Rocq features powerful metaprogramming facilities\markdownRendererSoftLineBreak
{}which allow writing custom \markdownRendererEmphasis{tactics} — programs\markdownRendererSoftLineBreak
{}written in a meta-language which generate terms of\markdownRendererSoftLineBreak
{}the target language, greatly simplifying the\markdownRendererSoftLineBreak
{}proofs.\markdownRendererSoftLineBreak
{}While set-theoretical semantics of Rocq assume the\markdownRendererSoftLineBreak
{}existence of countable hierarchy of inaccessible\markdownRendererSoftLineBreak
{}cardinals \markdownRendererCite{2}+{}{}{werner1997sets}+{}{}{timany2017consistency},\markdownRendererSoftLineBreak
{}our formalization never delves into the hierarchy\markdownRendererSoftLineBreak
{}and can be formalized in a much more restricted\markdownRendererSoftLineBreak
{}theory.\markdownRendererParagraphSeparator
{}To formalize our strong normalization\markdownRendererSoftLineBreak
{}proof, we first need to provide a general\markdownRendererSoftLineBreak
{}metatheory of PTS as laid out in Section \markdownRendererLink{\markdownRendererHash{}pts-metatheory}{\markdownRendererHash{}pts-metatheory}{#pts-metatheory}{}.\markdownRendererSoftLineBreak
{}We could reuse existing projects on PTS\markdownRendererSoftLineBreak
{}formalizations such as \markdownRendererCite{4}+{}{}{siles2012pure}+{}{}{adams2006pure}+{}{}{barras1997coqencoq}+{}{}{felicissimo2024generic}\markdownRendererSoftLineBreak
{}but we wanted to explore the topic ourselves, which\markdownRendererSoftLineBreak
{}led to a few minor differences discussed\markdownRendererSoftLineBreak
{}previously; thus we started from scratch.\markdownRendererParagraphSeparator
{}When one formalizes syntax with binders in a proof\markdownRendererSoftLineBreak
{}assistant, one of the most important decisions is\markdownRendererSoftLineBreak
{}how to represent binders.\markdownRendererSoftLineBreak
{}Here we went with making our terms indexed by the\markdownRendererSoftLineBreak
{}number of free variables it uses,\markdownRendererSoftLineBreak
{}a standard technique originating from \markdownRendererCite{3}+{}{}{bellegarde1994substitution}+{}{}{bird1999debruijn}+{}{}{altenkirch1999monadic}:\markdownRendererInterblockSeparator
{}\markdownRendererInputFencedCode{./_markdown_PTS-rk/e6f5e5ccf96f4f2b05491e864c2345ab.verbatim}{coq}{coq}\markdownRendererInterblockSeparator
{}(Inside our formalization project,\markdownRendererSoftLineBreak
{}this definition is located here: \proof{rocq:Syntax.v\#L7})\markdownRendererParagraphSeparator
{}Here, \markdownRendererInlineMath{L} is a type of sorts (\markdownRendererInlineMath{L} comes from\markdownRendererSoftLineBreak
{}\markdownRendererEmphasis{level} which is a name used for it in universe\markdownRendererSoftLineBreak
{}hierarchies), and \markdownRendererInlineMath{n} is a number of free\markdownRendererSoftLineBreak
{}variables.\markdownRendererSoftLineBreak
{}Entering a new scope (such as in \markdownRendererInlineMath{Π} and in\markdownRendererSoftLineBreak
{}\markdownRendererInlineMath{\mathtt{lam}}) adds one new free variable;\markdownRendererSoftLineBreak
{}variables, then, have to be numbers in range\markdownRendererSoftLineBreak
{}\markdownRendererInlineMath{[0;n)}, which are usually represented by a\markdownRendererSoftLineBreak
{}\markdownRendererInlineMath{\mathtt{Fin}} type.\markdownRendererParagraphSeparator
{}What is convenient in this representation is that\markdownRendererSoftLineBreak
{}it makes it harder to make mistakes when wrangling\markdownRendererSoftLineBreak
{}with context, defining substitution etc. — all\markdownRendererSoftLineBreak
{}the technical details which make working with\markdownRendererSoftLineBreak
{}syntax with binders so hard (authors of the first\markdownRendererSoftLineBreak
{}implementation of such calculi, AUTOMATH \markdownRendererCite{2}+{}{}{debruijn1972lambda}+{}{}{debruijn1970automath},\markdownRendererSoftLineBreak
{}remarked it being notoriously hard to get right).\markdownRendererSoftLineBreak
{}Also, it has a certain category-theoretical feel to\markdownRendererSoftLineBreak
{}it: sufficiently generalized operations of\markdownRendererSoftLineBreak
{}renaming and substitution give the typal\markdownRendererSoftLineBreak
{}representation of syntax a structure of a\markdownRendererSoftLineBreak
{}\markdownRendererEmphasis{relative monad} \markdownRendererCite{2}+{}{}{altenkirch2015monads}+{}{}{altenkirch1999monadic}.\markdownRendererSoftLineBreak
{}It is really hard to go wrong when implementing\markdownRendererSoftLineBreak
{}these operations (and the natural implementation\markdownRendererSoftLineBreak
{}turns variables into de Bruijn indices \markdownRendererCite{1}+{}{}{debruijn1972lambda});\markdownRendererSoftLineBreak
{}on top of them, the usual substitution is easily\markdownRendererSoftLineBreak
{}definable.\markdownRendererParagraphSeparator
{}Most of the proof development in Rocq mirrors what\markdownRendererSoftLineBreak
{}is written in Section \markdownRendererLink{\markdownRendererHash{}pts-metatheory}{\markdownRendererHash{}pts-metatheory}{#pts-metatheory}{}; in fact,\markdownRendererSoftLineBreak
{}we first wrote a proof in Rocq and then retold it\markdownRendererSoftLineBreak
{}in written English here, omitting tedious\markdownRendererSoftLineBreak
{}technicalities.\markdownRendererSoftLineBreak
{}The only considerable difference is in the handling of\markdownRendererSoftLineBreak
{}the set of strongly normalizing terms since\markdownRendererSoftLineBreak
{}the proof of Substitution lemma given\markdownRendererSoftLineBreak
{}above does not follow the structure of inductive\markdownRendererSoftLineBreak
{}definition \markdownRendererLink{\markdownRendererHash{}sn-definition}{\markdownRendererHash{}sn-definition}{#sn-definition}{}.\markdownRendererSoftLineBreak
{}Instead, we use a syntax-directed characterization\markdownRendererSoftLineBreak
{}of the set of strongly normalizing terms \markdownRendererCite{1}+{}{}{Joachimski2003Short}:\markdownRendererInterblockSeparator
{}\begin{definition}[Syntax-directed characterization of strong normalization \proof{rocq:SN.v\#L175}]Sets of terms \markdownRendererInlineMath{\mathsf{SN}} (strongly\markdownRendererSoftLineBreak
{}normalizing), \markdownRendererInlineMath{\mathsf{SNx}} (strongly normalizing\markdownRendererSoftLineBreak
{}neutral) and a relation \markdownRendererInlineMath{→_\mathsf {WHR}} (weak\markdownRendererSoftLineBreak
{}head reduction) are the least sets where the\markdownRendererSoftLineBreak
{}following rules are valid:\begin{mathpar}
  \inferrule{ }{\mathsf{SNx}\,s}(s ∈ \mathcal{S})
  \and
  \inferrule
    {\mathsf{SN}\,T \\ \mathsf{SN}\,U}
    {\mathsf{SNx}(Π (x : T).U)}
  \and
  \inferrule{ }{\mathsf{SNx}\,x}
  \and
  \inferrule
    {\mathsf{SNx}\,t \\ \mathsf{SN}\,u}
    {\mathsf{SNx}\,(t\,u)}
  \\
  \inferrule{\mathsf{SNx}\,t}{\mathsf{SN}\,t}
  \and
  \inferrule
    {\mathsf{SN}\,T \\ \mathsf{SN}\,t}
    {\mathsf{SN}\,(λ (x : T).t)}
  \and
  \inferrule
    {t →_\mathsf {WHR}u \\ \mathsf{SN}\,u}
    {\mathsf{SN}\,t}
  \and
  \inferrule
    {\mathsf{SN}\,U \\ \mathsf{SN}\,u}
    {(λ (x : U).t)\,u →_\mathsf {WHR} t[u/x]}
  \and
  \inferrule{t →_\mathsf {WHR}t'}
            {t\,u →_\mathsf {WHR}t'\,u}
\end{mathpar}\end{definition}\markdownRendererInterblockSeparator
{}In a sense, \markdownRendererInlineMath{\mathsf{SNx}} is a description of\markdownRendererSoftLineBreak
{}atomic strongly normalizing terms;\markdownRendererSoftLineBreak
{}\markdownRendererInlineMath{\mathsf{SN}} adds a case for lambdas and for weak\markdownRendererSoftLineBreak
{}head expansion, covering all possible cases.\markdownRendererSoftLineBreak
{}Existence of such sets is by virtue of them being\markdownRendererSoftLineBreak
{}subsets of the set of terms, so such mutual\markdownRendererSoftLineBreak
{}inductive definition is not problematic.\markdownRendererParagraphSeparator
{}One can prove that set \markdownRendererInlineMath{\mathsf{SN}} is equal to\markdownRendererSoftLineBreak
{}\markdownRendererInlineMath{\mathsf{sn}}; we, however, need only the\markdownRendererSoftLineBreak
{}following:\markdownRendererInterblockSeparator
{}\begin{lemma}[Soundness of $\mathsf{SN}$ \proof{rocq:SN.v\#L222}]\markdownRendererInlineMath{\mathsf{SN} ⊆ \mathsf{sn}}.\end{lemma}\markdownRendererInterblockSeparator
{}\begin{proof}By mutual induction on the evidence\markdownRendererSoftLineBreak
{}\markdownRendererInlineMath{\mathsf{SN}\,t} for arbitrary \markdownRendererInlineMath{t}.\markdownRendererSoftLineBreak
{}To do this, we also have to provide a term-indexed\markdownRendererSoftLineBreak
{}proposition we are proving for \markdownRendererInlineMath{\mathsf{SNx}} and\markdownRendererSoftLineBreak
{}a relation on terms we are proving for\markdownRendererSoftLineBreak
{}\markdownRendererInlineMath{\mathsf{WHR}}:\markdownRendererInterblockSeparator
{}\markdownRendererUlBegin
\markdownRendererUlItem \markdownRendererInlineMath{∀ t, \mathsf{SNx}\,t ⇒ \mathsf{sn}\,t}.\markdownRendererUlItemEnd 
\markdownRendererUlItem \markdownRendererInlineMath{∀ t\,t', \,(t →_\mathsf {WHR}t') ⇒ (t →_\mathsf {WHR - sn}t')}\markdownRendererSoftLineBreak
{}(where the last relation is defined in the same\markdownRendererSoftLineBreak
{}way as WHR, but the head contains evidence of\markdownRendererSoftLineBreak
{}\markdownRendererInlineMath{\mathsf{sn}} instead of \markdownRendererInlineMath{\mathsf{SN}}.)\markdownRendererUlItemEnd 
\markdownRendererUlEnd \markdownRendererInterblockSeparator
{}Proof is extremely simple: by Lemmas \markdownRendererLink{\markdownRendererHash{}sn-props}{\markdownRendererHash{}sn-props}{#sn-props}{},\markdownRendererLink{\markdownRendererHash{}sn-atomic}{\markdownRendererHash{}sn-atomic}{#sn-atomic}{},\markdownRendererLink{\markdownRendererHash{}sn-whe}{\markdownRendererHash{}sn-whe}{#sn-whe}{}.\markdownRendererSoftLineBreak
{}We just have to notice that if the term is\markdownRendererSoftLineBreak
{}\markdownRendererInlineMath{\mathsf{SNx}}, it is atomic,\markdownRendererSoftLineBreak
{}so lemma \markdownRendererLink{\markdownRendererHash{}sn-atomic}{\markdownRendererHash{}sn-atomic}{#sn-atomic}{} applies in the relevant\markdownRendererSoftLineBreak
{}case.\end{proof}\markdownRendererInterblockSeparator
{}Then, to prove strong normalization,\markdownRendererSoftLineBreak
{}it is enough to prove that if a term is typable,\markdownRendererSoftLineBreak
{}it is \markdownRendererInlineMath{\mathsf{SN}}; we only have to state\markdownRendererSoftLineBreak
{}Application Lemma and Substitution Lemma in terms\markdownRendererSoftLineBreak
{}of \markdownRendererInlineMath{\mathsf{SN}}, and the proof goes through.\markdownRendererInterblockSeparator
{}\begin{lemma}\markdownRendererInlineMath{\mathrm{AL}_s} for all \markdownRendererInlineMath{s ≤ r} implies\markdownRendererSoftLineBreak
{}\markdownRendererInlineMath{\mathrm{SL}_r}. \proof{rocq:Ranks.v\#L74}\end{lemma}\markdownRendererInterblockSeparator
{}\begin{proof}We have to prove that \markdownRendererInlineMath{\mathsf{SN}\,t} and\markdownRendererSoftLineBreak
{}\markdownRendererInlineMath{\mathsf{SN}\,u} imply \markdownRendererInlineMath{\mathsf{SN}\,(t[u/x])}.\markdownRendererSoftLineBreak
{}We do so by mutual induction on evidence for \markdownRendererInlineMath{t}\markdownRendererSoftLineBreak
{}with the following propositions proven for\markdownRendererSoftLineBreak
{}\markdownRendererInlineMath{\mathsf{SNx}} and WHR:\markdownRendererInterblockSeparator
{}\markdownRendererUlBegin
\markdownRendererUlItem \markdownRendererInlineMath{\mathsf{SNx}\,t ∧ Γ ⊢ t : T ⇒ \mathsf{SNx}\,(t[u/x]) ∨ (\mathsf{SN}\,(t[u/x]) ∧ ∃ s ≤ r. \;Γ ⊢ T : \rk{s})};\markdownRendererUlItemEnd 
\markdownRendererUlItem \markdownRendererInlineMath{(t →_\mathsf {WHR}t') ⇒ (t[u/x] →_\mathsf {WHR}t'[u/x]) ∧ (\mathsf{SN}\,(t'[u/x]) ⇒ \mathsf{SN}\,(t[u/x]))}.\markdownRendererUlItemEnd 
\markdownRendererUlEnd \markdownRendererInterblockSeparator
{}The proof then follows the non-formal treatment\markdownRendererSoftLineBreak
{}given above.\end{proof}\markdownRendererInterblockSeparator
{}
\markdownRendererSectionEnd 
\markdownRendererSectionEnd \markdownRendererSectionBegin
\markdownRendererHeaderAttributeContextBegin
\markdownRendererAttributeIdentifier{complexity}\markdownRendererHeadingOne{Computational complexity of \PTSrk{}}\markdownRendererHeaderAttributeContextEnd{}\markdownRendererInterblockSeparator
{}In this section, we prove that the convertibility\markdownRendererSoftLineBreak
{}and the type checking of closed, well-typed\markdownRendererSoftLineBreak
{}\PTSrk{} terms is\markdownRendererSoftLineBreak
{}\textsc{TOWER}-complete.\markdownRendererParagraphSeparator
{}(Proofs in this and following sections were not\markdownRendererSoftLineBreak
{}verified in Rocq by any of the authors,\markdownRendererSoftLineBreak
{}but there is a formalization of some results by\markdownRendererSoftLineBreak
{}GPT-6 Astra. Informal proofs, however,\markdownRendererSoftLineBreak
{}are relatively simple, non-technical, and their\markdownRendererSoftLineBreak
{}formalization requires many times more preliminary\markdownRendererSoftLineBreak
{}work than seems reasonable.)\markdownRendererParagraphSeparator
{}As requested further, LLM confessed that it was\markdownRendererSoftLineBreak
{}largely inspired by studies on the complexity of\markdownRendererSoftLineBreak
{}normalization for STLC \markdownRendererCite{4}+{}{}{statman1979typed}+{}{}{schwichtenberg1991upper}+{}{}{beckmann2001exact}+{}{}{nguyen2024simply}.\markdownRendererSoftLineBreak
{}And indeed, upon taking a closer look at the proof\markdownRendererSoftLineBreak
{}of Theorem \markdownRendererLink{\markdownRendererHash{}normalization}{\markdownRendererHash{}normalization}{#normalization}{}, one cannot help but\markdownRendererSoftLineBreak
{}notice it actually being a primitive recursive\markdownRendererSoftLineBreak
{}algorithm \markdownRendererCite{3}+{}{}{skolem1923}+{}{}{kleene1952introduction}+{}{}{peter1951rekursive}\markdownRendererSoftLineBreak
{}for term normalization in disguise.\markdownRendererSoftLineBreak
{}Or something a tad faster:\markdownRendererInterblockSeparator
{}\begin{definition}[$d$-bounded terms]A \PTSrk{} term is \markdownRendererInlineMath{d}-bounded if it is well-typed\markdownRendererSoftLineBreak
{}in some well-formed context, and the heads of all\markdownRendererSoftLineBreak
{}redexes in it have rank equal to or below \markdownRendererInlineMath{d}.\end{definition}\markdownRendererInterblockSeparator
{}\begin{theorem}[Elementary recursive normalization of bounded \PTSrk{} \formalized{rocq-ai:NormalizationComplexity/BoundedNormalization.v\#L19}]\label{bounded-sn}For each \markdownRendererInlineMath{d}, there is an elementary recursive\markdownRendererSoftLineBreak
{}algorithm \markdownRendererCite{1}+{}{}{rose1984subrecursion} for finding a\markdownRendererSoftLineBreak
{}normal form of \markdownRendererInlineMath{d}-bounded terms.\end{theorem}\markdownRendererInterblockSeparator
{}\begin{proof}Consider the following relation of “greedy\markdownRendererSoftLineBreak
{}parallel reduction”:\begin{mathpar}
  \inferrule{ }{x ⇒ x\\ s ⇒ s}\and
  \inferrule
    {T ⇒ T' \\ U ⇒ U'}
    {Π (x : T).U ⇒ Π (x : T').U'}\and
  \inferrule
    {T ⇒ T' \\ t ⇒ t'}
    {λ (x : T).t ⇒ λ (x : T').t'}\and
  \inferrule
    {t ⇒ (λ (x : T).t') \\ u ⇒ u'}
    {t\,u ⇒ t'[u'/x]}\and
  \inferrule
    {t ⇒ t' \\ t' ≠ λ … \\ u ⇒ u'}
    {t\,u ⇒ t'\,u'}
\end{mathpar}\markdownRendererInterblockSeparator
{}It differs from the usual parallel reduction since it\markdownRendererSoftLineBreak
{}greedily applies every “many-argument” redex\markdownRendererSoftLineBreak
{}available in a term.\markdownRendererSoftLineBreak
{}For example,\markdownRendererSoftLineBreak
{}\markdownRendererInlineMath{(λ (x : X).λ (y : Y).λ (z : Z).t)\,a\,b\,c} gets\markdownRendererSoftLineBreak
{}reduced to \markdownRendererInlineMath{t'[a/x][b/y][c/z]} in a single step.\markdownRendererSoftLineBreak
{}This relation has the following properties:\markdownRendererInterblockSeparator
{}\markdownRendererUlBegin
\markdownRendererUlItem \markdownRendererInlineMath{( ⇒ ) ⊆ ( →_β^∗ )}, so it is several usual\markdownRendererSoftLineBreak
{}reductions baked into one.\markdownRendererUlItemEnd 
\markdownRendererUlItem This relation is total and functional.\markdownRendererUlItemEnd 
\markdownRendererUlItem This relation can be realized with an elementary\markdownRendererSoftLineBreak
{}recursive algorithm \markdownRendererInlineMath{\mathcal{A}}:\markdownRendererSoftLineBreak
{}traverse the term bottom-up, immediately\markdownRendererSoftLineBreak
{}applying every \markdownRendererInlineMath{β}-redex which is formed when\markdownRendererSoftLineBreak
{}the result of a recursive run of the algorithm on\markdownRendererSoftLineBreak
{}the application head is a lambda.\markdownRendererUlItemEnd 
\markdownRendererUlItem If \markdownRendererInlineMath{t ⇒ t'}, then\markdownRendererSoftLineBreak
{}\markdownRendererInlineMath{|t'| ≤ (2 + ε )^{\mathcal{O}(|t|)}}.\markdownRendererUlItemEnd 
\markdownRendererUlItem If redex heads in \markdownRendererInlineMath{t} have ranks less than \markdownRendererEmphasis{or\markdownRendererSoftLineBreak
{}equal to} \markdownRendererInlineMath{d} and \markdownRendererInlineMath{t ⇒ t'}, redex heads in \markdownRendererInlineMath{t'}\markdownRendererSoftLineBreak
{}have ranks \markdownRendererEmphasis{less than} \markdownRendererInlineMath{d}: new redexes appear\markdownRendererSoftLineBreak
{}only after substitution of abstractions into\markdownRendererSoftLineBreak
{}application heads, and ranks of these\markdownRendererSoftLineBreak
{}abstractions are strictly less than the rank of\markdownRendererSoftLineBreak
{}redex heads in which these substitutions\markdownRendererSoftLineBreak
{}happened.\markdownRendererUlItemEnd 
\markdownRendererUlEnd \markdownRendererInterblockSeparator
{}Because of the last property, it is enough to run\markdownRendererSoftLineBreak
{}\markdownRendererInlineMath{\mathcal{A}} for \markdownRendererInlineMath{d} times in total to normalize\markdownRendererSoftLineBreak
{}a \markdownRendererInlineMath{d}-bounded term; since input size on \markdownRendererInlineMath{k}-th run\markdownRendererSoftLineBreak
{}is bounded by \markdownRendererInlineMath{k}-th tower of exponents,\markdownRendererSoftLineBreak
{}the whole algorithm is elementary recursive.\markdownRendererInterblockSeparator
{}To make sure that the corresponding function is\markdownRendererSoftLineBreak
{}total, output \markdownRendererInlineMath{0} if the result of an algorithm\markdownRendererSoftLineBreak
{}contains a \markdownRendererInlineMath{β}-redex.\end{proof}\markdownRendererInterblockSeparator
{}Similar result was already proven for STLC \markdownRendererCite{2}+{}{}{schwichtenberg1991upper}+{}{}{beckmann2001exact}\markdownRendererSoftLineBreak
{}since simple types can be classified by ranks as\markdownRendererSoftLineBreak
{}well.\markdownRendererSoftLineBreak
{}And since rank is easily computable for simple\markdownRendererSoftLineBreak
{}types, normalization of arbitrary STLC terms is in\markdownRendererSoftLineBreak
{}\markdownRendererInlineMath{\mathsf{TOWER}}.\markdownRendererSoftLineBreak
{}A little weaker result holds for our dependently\markdownRendererSoftLineBreak
{}typed PTS:\markdownRendererInterblockSeparator
{}\begin{theorem}\label{tower-convertibility}Convertibility of \markdownRendererStrongEmphasis{closed} well-typed \PTSrk{}\markdownRendererSoftLineBreak
{}terms is in \markdownRendererInlineMath{\mathsf{TOWER}}.\end{theorem}\markdownRendererInterblockSeparator
{}\begin{proof}By Theorem \markdownRendererLink{\markdownRendererHash{}bounded-sn}{\markdownRendererHash{}bounded-sn}{#bounded-sn}{}, it is enough to compute\markdownRendererSoftLineBreak
{}some upper bound \markdownRendererInlineMath{B} on the rank of types\markdownRendererSoftLineBreak
{}appearing in abstractions of a term,\markdownRendererSoftLineBreak
{}then run a normalizer for \markdownRendererInlineMath{(B + 1)}-bounded terms\markdownRendererSoftLineBreak
{}and compare the results; \markdownRendererInlineMath{B} has to be bound by an\markdownRendererSoftLineBreak
{}elementary recursive function on the size of\markdownRendererSoftLineBreak
{}input. We present such a function in\markdownRendererSoftLineBreak
{}Figure \markdownRendererLink{\markdownRendererHash{}rank-bound-algorithm}{\markdownRendererHash{}rank-bound-algorithm}{#rank-bound-algorithm}{}:\markdownRendererSoftLineBreak
{}\markdownRendererInlineMath{\mathtt{estimateRank}\;t} provides a bound on\markdownRendererSoftLineBreak
{}ranks of redex arguments in \markdownRendererInlineMath{t}.\markdownRendererInterblockSeparator
{}\begin{figure}[htb]\markdownRendererInputFencedCode{./_markdown_PTS-rk/896db171c5359b6c5edc4059aa93c517.verbatim}{agda}{agda}\caption{Computing a bound on ranks of redex arguments}\label{rank-bound-algorithm}\end{figure}\markdownRendererInterblockSeparator
{}The function \markdownRendererInlineMath{\mathtt{maxRankLiteral}\;t} computes\markdownRendererSoftLineBreak
{}maximal rank which is contained as a subterm in\markdownRendererSoftLineBreak
{}the term \markdownRendererInlineMath{t}.\markdownRendererSoftLineBreak
{}The workhorse of this algorithm,\markdownRendererSoftLineBreak
{}\markdownRendererInlineMath{\mathtt{traversal}\;Γ \;t\;d},\markdownRendererSoftLineBreak
{}accepts the list \markdownRendererInlineMath{Γ} of ranks in context of a term\markdownRendererSoftLineBreak
{}\markdownRendererInlineMath{t} and a “default rank” \markdownRendererInlineMath{d},\markdownRendererSoftLineBreak
{}and outputs the triple \markdownRendererInlineMath{(v, r, a)},\markdownRendererSoftLineBreak
{}where \markdownRendererInlineMath{v} is value of \markdownRendererInlineMath{t} as a rank (if supplied\markdownRendererSoftLineBreak
{}with enough arguments), \markdownRendererInlineMath{r} is its rank if \markdownRendererInlineMath{t} is\markdownRendererSoftLineBreak
{}a type (if supplied with enough arguments) and \markdownRendererInlineMath{a}\markdownRendererSoftLineBreak
{}is an upper bound on rank of types appearing in\markdownRendererSoftLineBreak
{}abstractions which we are trying to compute.\markdownRendererSoftLineBreak
{}The only suspicious case is\markdownRendererSoftLineBreak
{}\markdownRendererInlineMath{\mathtt{traversal}\,Γ \,x\,d} where we output \markdownRendererInlineMath{d}\markdownRendererSoftLineBreak
{}(the “default” rank) as a value of \markdownRendererInlineMath{x}:\markdownRendererSoftLineBreak
{}could we compute the value of \markdownRendererInlineMath{x} instead,\markdownRendererSoftLineBreak
{}if present?\markdownRendererSoftLineBreak
{}Regrettably, this would require normalization of a\markdownRendererSoftLineBreak
{}term corresponding to \markdownRendererInlineMath{x}, which is too long;\markdownRendererSoftLineBreak
{}instead we notice that \markdownRendererInlineMath{\rk{r}} can be reduced\markdownRendererSoftLineBreak
{}only from terms which already contain it as a\markdownRendererSoftLineBreak
{}subterm, so the value of \markdownRendererInlineMath{x} as a rank can only be\markdownRendererSoftLineBreak
{}\markdownRendererInlineMath{\rk{r}} which is already present somewhere in\markdownRendererSoftLineBreak
{}the original term.\markdownRendererSoftLineBreak
{}And our result is monotone with respect to the\markdownRendererSoftLineBreak
{}value we output here, so it is enough to provide\markdownRendererSoftLineBreak
{}some upper bound on it; thus it is enough to\markdownRendererSoftLineBreak
{}output a maximal rank which occurs in the original\markdownRendererSoftLineBreak
{}term as a subterm.\end{proof}\markdownRendererInterblockSeparator
{}Closedness is essential: without the context,\markdownRendererSoftLineBreak
{}we cannot determine the ranks of free type\markdownRendererSoftLineBreak
{}variables and hence cannot bound\markdownRendererSoftLineBreak
{}redex ranks.\markdownRendererInterblockSeparator
{}\begin{theorem}Convertibility of closed well-typed \PTSrk{} terms\markdownRendererSoftLineBreak
{}is \textsc{TOWER}-hard (hence\markdownRendererSoftLineBreak
{}\textsc{TOWER}-complete). \formalized{rocq-ai:NormalizationComplexity/STLCEmbedding.v\#L198}\markdownRendererSoftLineBreak
{}\markdownRendererNote{The badge certifies the annotation-preserving\markdownRendererSoftLineBreak
{}reduction from closed Church-style STLC,\markdownRendererSoftLineBreak
{}including typing, preservation and reflection of\markdownRendererSoftLineBreak
{}conversion, and the exact additive encoding-size\markdownRendererSoftLineBreak
{}bound.}\end{theorem}\markdownRendererInterblockSeparator
{}\begin{proof}By reduction from convertibility of Curry-style\markdownRendererSoftLineBreak
{}STLC terms, since that problem is\markdownRendererSoftLineBreak
{}\textsc{TOWER}-hard \markdownRendererCite{1}+{}{}{nguyen2024simply}.\markdownRendererInterblockSeparator
{}Just in case, we start from Curry-style STLC\markdownRendererSoftLineBreak
{}terms.\markdownRendererSoftLineBreak
{}We then reduce this task to convertibility of\markdownRendererSoftLineBreak
{}\markdownRendererStrongEmphasis{closed} Curry-style STLC terms by adding\markdownRendererSoftLineBreak
{}\markdownRendererInlineMath{λ}-abstractions to close over free variables;\markdownRendererSoftLineBreak
{}finally, we infer types of variables in an STLC\markdownRendererSoftLineBreak
{}term by a linear unification algorithm referenced\markdownRendererSoftLineBreak
{}in \markdownRendererCite{1}+{}{}{nguyen2024simply} and settle on \markdownRendererStrongEmphasis{closed}\markdownRendererSoftLineBreak
{}well-typed \markdownRendererStrongEmphasis{Church-style} STLC terms.\markdownRendererSoftLineBreak
{}The following translation then yields \markdownRendererStrongEmphasis{open}\markdownRendererSoftLineBreak
{}well-typed \PTSrk{} terms:\begin{align*}
  ⦇ x ⦈ & = x \\
          ⦇ λ x^σ.t
          ⦈ & = λ (x : σ).
              ⦇ t ⦈ \\
          ⦇ t\,u ⦈ & = ⦇ t ⦈ \,
          ⦇ u ⦈
\end{align*}\markdownRendererInterblockSeparator
{}The problem is that types \markdownRendererInlineMath{σ} might contain type\markdownRendererSoftLineBreak
{}variables \markdownRendererInlineMath{T_1, …, T_n} which are not bound.\markdownRendererSoftLineBreak
{}However, it is easily fixable by introducing\markdownRendererSoftLineBreak
{}missing types via abstractions:\begin{mathpar}
  t ↝ λ (T_1 : \rk{0})…
     (T_n : \rk{0}).⦇ t ⦈
\end{mathpar}\markdownRendererParagraphSeparator
{}Assigning all type variables rank \markdownRendererInlineMath{\rk{0}}\markdownRendererSoftLineBreak
{}preserves well-typedness since values of such\markdownRendererSoftLineBreak
{}types were never treated as functions in the original\markdownRendererSoftLineBreak
{}term.\end{proof}\markdownRendererInterblockSeparator
{}\begin{corollary}Type checking for closed \PTSrk{} terms is\markdownRendererSoftLineBreak
{}\textsc{TOWER}-complete.\end{corollary}\markdownRendererInterblockSeparator
{}\begin{proof}The following term demonstrates that\markdownRendererSoftLineBreak
{}convertibility of closed well-typed terms\markdownRendererSoftLineBreak
{}\markdownRendererInlineMath{t, t' : T} reduces to type checking:\begin{mathpar}
  λ (P : T → \rk{0})(x : P\,t).
  (λ (y : P\,t').y)\,x
\end{mathpar}\markdownRendererInterblockSeparator
{}This term is well-typed iff \markdownRendererInlineMath{P\,t =_β P\,t'} which\markdownRendererSoftLineBreak
{}holds iff \markdownRendererInlineMath{t =_β t'}.\markdownRendererInterblockSeparator
{}Therefore, we only need to prove that type checking\markdownRendererSoftLineBreak
{}for \PTSrk{} is in \textsc{TOWER}.\markdownRendererSoftLineBreak
{}Consider the following syntax-directed type\markdownRendererSoftLineBreak
{}synthesis rules:\begin{mathpar}
  \inferrule[Ax']{ }
  {Γ ⊢ \rk{r} ⇒ \rk{0}}
  \and
  \inferrule[Rule']
    {Γ ⊢ T ⇒ \rk{r} \\ Γ, x : T ⊢ U ⇒ \rk{s}}
    {Γ ⊢ Π (x : T).\,U ⇒ \rk{\max(r + 1, s)}}
  \\
  \inferrule[Var']{ }
  {Γ ⊢ x ⇒ T}x : T ∈ Γ
  \and
  \inferrule[Abs']
    {Γ ⊢ T ⇒ \rk{r} \\ Γ, x : T ⊢ u ⇒ U}
    {Γ ⊢ λ (x : T).u ⇒ Π (x : T).U}
  \and
  \inferrule[App']
    {Γ ⊢ u ⇒ Π (x : T).U \\ Γ ⊢ t ⇒ T' \\ T =_β T'}
    {Γ ⊢ u\,t ⇒ U[t/x]}
\end{mathpar}\markdownRendererInterblockSeparator
{}This relation is functional (in the sense that if\markdownRendererSoftLineBreak
{}\markdownRendererInlineMath{Γ ⊢ t : T} and \markdownRendererInlineMath{Γ ⊢ t : T'},\markdownRendererSoftLineBreak
{}then \markdownRendererInlineMath{T = T'} as words), and the corresponding\markdownRendererSoftLineBreak
{}partial function is realized by an algorithm in\markdownRendererSoftLineBreak
{}\textsc{TOWER}: this algorithm recurses on the\markdownRendererSoftLineBreak
{}term \markdownRendererInlineMath{t}, matches on the inferred types as needed\markdownRendererSoftLineBreak
{}and collects equalities on types it has to check.\markdownRendererSoftLineBreak
{}Then, it can naïvely check all equalities in bulk:\markdownRendererSoftLineBreak
{}there are at most \markdownRendererInlineMath{\mathcal{O}(|t|)} such\markdownRendererSoftLineBreak
{}equalities on terms of size\markdownRendererSoftLineBreak
{}\markdownRendererInlineMath{\mathcal{O}(2^{\mathrm{poly}(|t|)})},\markdownRendererSoftLineBreak
{}so this runs in \textsc{TOWER} by Theorem \markdownRendererLink{\markdownRendererHash{}tower-convertibility}{\markdownRendererHash{}tower-convertibility}{#tower-convertibility}{}.\markdownRendererSoftLineBreak
{}To finish the proof, it is enough to notice that\begin{mathpar}
  Γ ⊢ t : T \quad ⟺ \quad
  ∃ T', r.\;(Γ ⊢ T : r) ∧ (T =_β T') ∧ (Γ ⊢ t ⇒ T').
\end{mathpar}\markdownRendererInterblockSeparator
{}So the main algorithm infers the type of \markdownRendererInlineMath{T},\markdownRendererSoftLineBreak
{}checks that it is \markdownRendererInlineMath{\rk{r}} for some \markdownRendererInlineMath{r},\markdownRendererSoftLineBreak
{}infers the type of \markdownRendererInlineMath{t} and compares it with \markdownRendererInlineMath{T}.\end{proof}\markdownRendererInterblockSeparator
{}
\markdownRendererSectionEnd \markdownRendererSectionBegin
\markdownRendererHeaderAttributeContextBegin
\markdownRendererAttributeIdentifier{computational}\markdownRendererHeadingOne{\PTSrk{} as a computational model}\markdownRendererHeaderAttributeContextEnd{}\markdownRendererInterblockSeparator
{}In this section we prove that functions\markdownRendererSoftLineBreak
{}representable in \PTSrk{} are,\markdownRendererSoftLineBreak
{}under minor assumptions, precisely elementary\markdownRendererSoftLineBreak
{}recursive functions, and provide a bunch of other,\markdownRendererSoftLineBreak
{}more specialized, results.\markdownRendererParagraphSeparator
{}Another consequence of Theorem \markdownRendererLink{\markdownRendererHash{}bounded-sn}{\markdownRendererHash{}bounded-sn}{#bounded-sn}{} is a\markdownRendererSoftLineBreak
{}natural upper bound for functions \markdownRendererEmphasis{uniformly}\markdownRendererSoftLineBreak
{}representable in \PTSrk{} (this result is also\markdownRendererSoftLineBreak
{}similar to the one for STLC \markdownRendererCite{3}+{}{}{schwichtenberg1976definierbare}+{}{}{statman1979typed}+{}{}{fortune1983expressiveness}).\markdownRendererInterblockSeparator
{}\begin{definition}A family of functions on natural numbers\markdownRendererSoftLineBreak
{}\markdownRendererInlineMath{\mathcal{F}} is \markdownRendererEmphasis{uniformly} representable in\markdownRendererSoftLineBreak
{}\PTSrk{} if there is an \markdownRendererEmphasis{input natural number\markdownRendererSoftLineBreak
{}type} \markdownRendererInlineMath{N_I : \rk{r}} and an \markdownRendererEmphasis{output natural number\markdownRendererSoftLineBreak
{}type} \markdownRendererInlineMath{N_O : \rk{s}} together with \markdownRendererEmphasis{encodings}\markdownRendererSoftLineBreak
{}\markdownRendererInlineMath{∀ n ∈ ℕ.[n]^I : N_I} and \markdownRendererInlineMath{∀ n.[n]^O : N_O} such\markdownRendererSoftLineBreak
{}that all of them are in normal form and for each\markdownRendererSoftLineBreak
{}\markdownRendererInlineMath{(f : ℕ^k → ℕ) ∈ \mathcal{F}} there is a\markdownRendererSoftLineBreak
{}\markdownRendererEmphasis{representation} \markdownRendererInlineMath{F : N_I → … → N_I → N_O} which\markdownRendererSoftLineBreak
{}is also in normal form and such that\begin{mathpar}
  ∀ (n_1, …, n_k) ∈ ℕ^k : \quad F\,[n_1]^I\,…\,[n_k]^I =_β [f(n_1, \,…, \,n_k)]^O
\end{mathpar}\end{definition}\markdownRendererInterblockSeparator
{}\begin{remark}Note that the requirement for terms to be in normal\markdownRendererSoftLineBreak
{}form is not restrictive: we can always normalize a\markdownRendererSoftLineBreak
{}non-normal term.\end{remark}\markdownRendererInterblockSeparator
{}\begin{theorem}\label{uniform-elementary}A family of elementary recursive functions is\markdownRendererSoftLineBreak
{}\markdownRendererStrongEmphasis{not} uniformly representable in \PTSrk{}.\end{theorem}\markdownRendererInterblockSeparator
{}\begin{proof}By a direct diagonalization argument.\markdownRendererSoftLineBreak
{}Suppose we could encode all elementary recursive\markdownRendererSoftLineBreak
{}functions with \markdownRendererInlineMath{N_I : \rk{r}} and \markdownRendererInlineMath{N_O : \rk{s}}.\markdownRendererInterblockSeparator
{}By Theorem \markdownRendererLink{\markdownRendererHash{}bounded-sn}{\markdownRendererHash{}bounded-sn}{#bounded-sn}{}, the normalization of\markdownRendererSoftLineBreak
{}\markdownRendererInlineMath{(r + 1)}-bounded terms \markdownRendererInlineMath{\mathcal{SN}_{r + 1}} is\markdownRendererSoftLineBreak
{}elementary recursive.\markdownRendererSoftLineBreak
{}Therefore, the following function is also\markdownRendererSoftLineBreak
{}elementary recursive:\begin{mathpar}
  h(t) =
  \begin{cases}
    1, & \mathcal{SN}_{r + 1}(t\,[\mathtt{t}]^I) = [0]^O, \\
    0 & \text{otherwise}.
  \end{cases}
\end{mathpar}\markdownRendererParagraphSeparator
{}(Where \markdownRendererInlineMath{\mathtt{t} ∈ ℕ} is code of a term \markdownRendererInlineMath{t} used\markdownRendererSoftLineBreak
{}in the definition of \markdownRendererInlineMath{h}.)\markdownRendererInterblockSeparator
{}By assumption, it is encodable by \markdownRendererInlineMath{H : N_I → N_O}.\markdownRendererSoftLineBreak
{}Observe the value of \markdownRendererInlineMath{h(H)}:\begin{mathpar}
  h(H) = 1 ⇔ \mathcal{SN}_{r + 1}(H\,[\mathtt{H}]^I) = [0]^O ⇔ H\,[\mathtt{H}]^I =_β [0]^O ⇔ h(H) = 0
\end{mathpar}\markdownRendererParagraphSeparator
{}The first equivalence is by definition of \markdownRendererInlineMath{h};\markdownRendererSoftLineBreak
{}the second is by virtue of expression under\markdownRendererSoftLineBreak
{}\markdownRendererInlineMath{\mathcal{SN}} being \markdownRendererInlineMath{(r + 1)}-bounded;\markdownRendererSoftLineBreak
{}the last is by definition of \markdownRendererInlineMath{h} being encoded by\markdownRendererSoftLineBreak
{}\markdownRendererInlineMath{H}.\end{proof}\markdownRendererInterblockSeparator
{}The next logical question is:\markdownRendererSoftLineBreak
{}which functions are computable?\markdownRendererSoftLineBreak
{}Elementary recursive!\markdownRendererSoftLineBreak
{}Though, of course, not in the same sense as\markdownRendererSoftLineBreak
{}before.\markdownRendererInterblockSeparator
{}\begin{definition}[Polymorphic numeral family]For each fixed \markdownRendererInlineMath{j ≥ 0}, there is a type of\markdownRendererSoftLineBreak
{}\markdownRendererEmphasis{polymorphic Church numerals} of rank \markdownRendererInlineMath{j + 2}:\begin{mathpar}
  ℕ_j ≜ Π (X : \rk{j}).
         (X → X) → X → X;
\end{mathpar}\markdownRendererInterblockSeparator
{}and for each fixed \markdownRendererInlineMath{n}, there is a corresponding\markdownRendererSoftLineBreak
{}\markdownRendererEmphasis{Church literal} of type \markdownRendererInlineMath{ℕ_j}:\begin{mathpar}
  \overline{n} ≜ λ \,X\,f\,x.f^n\,x.
\end{mathpar}\markdownRendererInterblockSeparator
{}(\markdownRendererInlineMath{f^n\,x} means repeated application of \markdownRendererInlineMath{f} to\markdownRendererSoftLineBreak
{}\markdownRendererInlineMath{x}: \markdownRendererInlineMath{f(…(f\,x)…)}.)\end{definition}\markdownRendererInterblockSeparator
{}\begin{theorem}[Non-uniform elementary representability]\label{non-uniform-rep}For each elementary recursive function\markdownRendererSoftLineBreak
{}\markdownRendererInlineMath{f : ℕ^n → ℕ}, there is a \markdownRendererEmphasis{cost} \markdownRendererInlineMath{c_f} such that\markdownRendererSoftLineBreak
{}for each \markdownRendererEmphasis{base rank} \markdownRendererInlineMath{j}, there is a\markdownRendererSoftLineBreak
{}\markdownRendererEmphasis{representation}\markdownRendererSoftLineBreak
{}\markdownRendererInlineMath{F_j : ℕ_{j + c_f} → … → ℕ_{j + c_f} → ℕ_j}\markdownRendererSoftLineBreak
{}such that\begin{mathpar}
  ∀ (x_1, …, x_n) ∈ ℕ^n : \quad
  F_j\,\overline{x_1}\,…\,\overline{x_n} =_β \overline{f(x_1, …, x_n)}.
\end{mathpar}\end{theorem}\markdownRendererInterblockSeparator
{}\begin{proof}As is well-known, simply typed Church numerals\markdownRendererSoftLineBreak
{}admit addition and multiplication \markdownRendererCite{2}+{}{}{church1941calculi}+{}{}{barendregt1984lambda}.\markdownRendererSoftLineBreak
{}It is still true for our rank-refined version:\begin{mathpar}
  \begin{array}{rcll}
    \mathsf{Add}_j & ≜ &
    λ \,m\,n\,X\,f\,x.
    m\,X\,f\,(n\,X\,f\,x) & : ℕ_j → ℕ_j → ℕ_j \\
    \mathsf{Mul}_j & ≜ &
    λ \,m\,n\,X\,f\,x.
    m\,X\,(n\,X\,f)\,x & : ℕ_j → ℕ_j → ℕ_j
  \end{array}
\end{mathpar}\markdownRendererInterblockSeparator
{}By substitution-basis theorem \markdownRendererCite{1}+{}{}{prunescu2026minimal},\markdownRendererSoftLineBreak
{}it is enough to additionally provide modulus\markdownRendererSoftLineBreak
{}\markdownRendererInlineMath{x\,\mathrm{mod}\,y} (where\markdownRendererSoftLineBreak
{}\markdownRendererInlineMath{x\,\mathrm{mod}\,0 = x}) and binary\markdownRendererSoftLineBreak
{}exponentiation \markdownRendererInlineMath{2^x} since substitution,\markdownRendererSoftLineBreak
{}constants and projections are “free” in our\markdownRendererSoftLineBreak
{}calculus.\markdownRendererInterblockSeparator
{}By Theorem \markdownRendererLink{\markdownRendererHash{}uniform-elementary}{\markdownRendererHash{}uniform-elementary}{#uniform-elementary}{},\markdownRendererSoftLineBreak
{}we cannot hope to have both remaining operations\markdownRendererSoftLineBreak
{}represented uniformly.\markdownRendererSoftLineBreak
{}Indeed, naïve import of traditional exponentiation\markdownRendererSoftLineBreak
{}term for untyped Church numerals \markdownRendererCite{2}+{}{}{church1941calculi}+{}{}{rosser1984highlights}\markdownRendererSoftLineBreak
{}raises the rank of an exponent:\begin{multline*}
  \mathsf{Exp}_j ≜ λ \,m\,n\,(X : \rk{j})\,f\,x.\,
          n\,(X → X)\,(λ (h : X → X)(y : X).
              m\,X\,h\,y
          )\,f\,x
          \\
          : ℕ_j → ℕ_{j + 1} → ℕ_j
\end{multline*}\markdownRendererInterblockSeparator
{}And that is exactly the reason we allowed input\markdownRendererSoftLineBreak
{}and output ranks of Church numerals to differ.\markdownRendererInterblockSeparator
{}There still remains a problem, however:\markdownRendererSoftLineBreak
{}by allowing ranks to differ, we lost uniformity of\markdownRendererSoftLineBreak
{}substitution.\markdownRendererSoftLineBreak
{}In particular, consider an elementary function\markdownRendererSoftLineBreak
{}\markdownRendererInlineMath{x ↦ x + 2^x}.\markdownRendererSoftLineBreak
{}To represent such computation using combinators we\markdownRendererSoftLineBreak
{}defined, \markdownRendererInlineMath{x} has to have both type \markdownRendererInlineMath{ℕ_j} (as\markdownRendererSoftLineBreak
{}left addend) and \markdownRendererInlineMath{ℕ_{j + 1}} (as\markdownRendererSoftLineBreak
{}exponentiation power).\markdownRendererSoftLineBreak
{}Instead of introducing subtyping in the system,\markdownRendererSoftLineBreak
{}we make use of the following family of\markdownRendererSoftLineBreak
{}\markdownRendererEmphasis{numeral-preserving lowerings}:\begin{mathpar}
  δ_j ≜ λ \,n\,X\,f\,x.\,
          n\,(X → X)\,(λ \,h\,y.f\,(h\,y))\,
          (λ \,y.y)\,x : ℕ_{j + 1} → ℕ_j
\end{mathpar}\markdownRendererInterblockSeparator
{}As the name implies, the goal of such combinators\markdownRendererSoftLineBreak
{}is to lower the rank of a numeral,\markdownRendererSoftLineBreak
{}but preserve its “value”.\markdownRendererSoftLineBreak
{}It is achieved in the following way:\markdownRendererSoftLineBreak
{}given an \markdownRendererInlineMath{f} and \markdownRendererInlineMath{x} for lower-ranked type,\markdownRendererSoftLineBreak
{}an original combinator for \markdownRendererInlineMath{n} is used to build an\markdownRendererSoftLineBreak
{}\markdownRendererInlineMath{n}-fold composition of \markdownRendererInlineMath{f} with itself, which is\markdownRendererSoftLineBreak
{}then applied to \markdownRendererInlineMath{x}.\markdownRendererInterblockSeparator
{}Thus, we can lower ranks arbitrarily:\markdownRendererSoftLineBreak
{}\markdownRendererInlineMath{δ_{b, a} ≜ δ_a ∘ … ∘ δ_{b - 1} : ℕ_b → ℕ_a}.\markdownRendererSoftLineBreak
{}This restores substitution and allows gluing\markdownRendererSoftLineBreak
{}representation of elementary recursive functions\markdownRendererSoftLineBreak
{}together; cost \markdownRendererInlineMath{c_f} of a function can be computed\markdownRendererSoftLineBreak
{}from any expression that represents it — addition\markdownRendererSoftLineBreak
{}and multiplication cost \markdownRendererInlineMath{0}, exponentiation costs\markdownRendererSoftLineBreak
{}\markdownRendererInlineMath{1}, operands are merged by taking the maximum of\markdownRendererSoftLineBreak
{}their costs.\markdownRendererSoftLineBreak
{}(Of course, cost is not unique and depends on the\markdownRendererSoftLineBreak
{}way the function is represented with an expression.)\markdownRendererInterblockSeparator
{}All that is left, then, is to represent modulus.\markdownRendererSoftLineBreak
{}The general idea is as follows: first,\markdownRendererSoftLineBreak
{}suppose that we are dealing with a simpler problem\markdownRendererSoftLineBreak
{}of taking modulus by a constant modulo,\markdownRendererSoftLineBreak
{}e.g. \markdownRendererInlineMath{\mathrm{mod}\,5}.\markdownRendererSoftLineBreak
{}Notice that there are only \markdownRendererInlineMath{5} possible answers:\markdownRendererSoftLineBreak
{}\markdownRendererInlineMath{0},\markdownRendererInlineMath{1},\markdownRendererInlineMath{2},\markdownRendererInlineMath{3},\markdownRendererInlineMath{4}.\markdownRendererSoftLineBreak
{}As such, they can be represented by the type which\markdownRendererSoftLineBreak
{}contains exactly \markdownRendererInlineMath{5} elements;\markdownRendererSoftLineBreak
{}Church encoding \markdownRendererCite{1}+{}{}{bohm1985automatic} yields a type\markdownRendererSoftLineBreak
{}of \markdownRendererEmphasis{selectors}:\begin{mathpar}
  V_5 ≜ Π (X : \rk{j}).\,
          X → X → X → X → X → X : \rk{j + 1}
\end{mathpar}\markdownRendererInterblockSeparator
{}Its \markdownRendererInlineMath{5} different canonical values simply pick one\markdownRendererSoftLineBreak
{}of \markdownRendererInlineMath{5} inputs, to each their own.\markdownRendererInterblockSeparator
{}There is an operation of \markdownRendererEmphasis{cyclic shift} on\markdownRendererSoftLineBreak
{}selectors:\begin{mathpar}
  c_5 ≜ λ \,s\,X\,x_1\,x_2\,x_3\,x_4
      \,x_5.\,s\,X\,x_2\,x_3\,x_4\,x_5\,x_1 : V_5 → V_5
\end{mathpar}\markdownRendererInterblockSeparator
{}Iterating it \markdownRendererInlineMath{n} times on some pinned “zeroth”\markdownRendererSoftLineBreak
{}selector, we actually get the\markdownRendererSoftLineBreak
{}\markdownRendererInlineMath{(n\,\mathrm{mod}\,5)}-th selector.\markdownRendererSoftLineBreak
{}All that is left is to read the numeral back:\begin{mathpar}
  ρ ≜ λ \,s\,X\,f\,x.\,s\,X\,
      (f^0\,x)\,(f^1\,x)\,(f^2\,x)\,(f^3\,x)\,(f^4\,x) : V_5 → ℕ_j
\end{mathpar}\markdownRendererInterblockSeparator
{}Now, to model \markdownRendererInlineMath{m\,\mathrm{mod}\,n} in\markdownRendererSoftLineBreak
{}this vein, we need to \markdownRendererEmphasis{generate} \markdownRendererInlineMath{V_n}\markdownRendererSoftLineBreak
{}dynamically.\markdownRendererSoftLineBreak
{}Lucky for us, \PTSrk{} is dependently typed so\markdownRendererSoftLineBreak
{}there is hope.\markdownRendererInterblockSeparator
{}Indeed, the following construction was\markdownRendererSoftLineBreak
{}produced in a discussion with GPT-6 Astra.\markdownRendererInterblockSeparator
{}We are going to produce a combinator\markdownRendererSoftLineBreak
{}\markdownRendererInlineMath{\mathsf{Mod} : ℕ_{j + 1} → ℕ_{j + 4} → ℕ_j}.\markdownRendererSoftLineBreak
{}Fix parameter \markdownRendererInlineMath{X : \rk{j}} and consider the\markdownRendererSoftLineBreak
{}following signature for abstract \markdownRendererInlineMath{S : \rk{j + 1}}:\begin{mathpar}
  \begin{array}{ll}
    K : X → S, & z : S, \\
    L : (X → S) → X → S, & c : S → S, \\
    ρ : S → (X → X) → X → X,
    & r : (X → X) → (X → X).
  \end{array}
\end{mathpar}\markdownRendererInterblockSeparator
{}Intention is that \markdownRendererInlineMath{S} is simply-typed \markdownRendererInlineMath{V_q} for\markdownRendererSoftLineBreak
{}\markdownRendererInlineMath{X}, \markdownRendererInlineMath{z} is its “zero”, \markdownRendererInlineMath{K} is a “zero” for\markdownRendererSoftLineBreak
{}\markdownRendererInlineMath{V_{q + 1}}, \markdownRendererInlineMath{c} is a cyclic shift for \markdownRendererInlineMath{V_q},\markdownRendererSoftLineBreak
{}\markdownRendererInlineMath{L} is a cyclic shift for \markdownRendererInlineMath{V_{q + 1}},\markdownRendererSoftLineBreak
{}\markdownRendererInlineMath{ρ} converts a selector into a numeral,\markdownRendererSoftLineBreak
{}and \markdownRendererInlineMath{r} is the result of modulus computed using\markdownRendererSoftLineBreak
{}structure from the previous step (this is similar\markdownRendererSoftLineBreak
{}to the trick used to define a predecessor\markdownRendererSoftLineBreak
{}function).\markdownRendererInterblockSeparator
{}A callback accepting such a structure and returning\markdownRendererSoftLineBreak
{}\markdownRendererInlineMath{X} has the following type:\begin{align*}
  H_X ≜ \; & Π (S : \rk{j + 1}).\;
          (X → S) → S → ((X → S) → X → S) → (S → S) → \\ & (S → (X → X) → X → X) → ((X → X) → X → X) → X
          \quad : \rk{j + 3}
\end{align*}\markdownRendererInterblockSeparator
{}A type of “records” is thus\markdownRendererSoftLineBreak
{}\markdownRendererInlineMath{P_X ≜ H_X → X : \rk{j + 4}},\markdownRendererSoftLineBreak
{}with a constructor\begin{mathpar}
  ⟨ S, K, z, L, c, ρ, r⟩ ≜ λ h.\,h\,S\,K\,z\,L\,c\,ρ \,r.
\end{mathpar}\markdownRendererInterblockSeparator
{}As an initial \markdownRendererInlineMath{P_X}, we build a structure for\markdownRendererSoftLineBreak
{}\markdownRendererInlineMath{V_1} and an \markdownRendererInlineMath{r} which behaves like\markdownRendererSoftLineBreak
{}\markdownRendererInlineMath{m\,\mathrm{mod}\,0 = m}:\begin{align*}
  S_0 & ≜ X → X, \\
          K_0 & ≜ λ (a : X)(b : X).a, &
          z_0 & ≜ λ (a : X).a, \\
          L_0 & ≜ λ (t : X → S_0)(a : X)(b : X).
              t\,b\,a, &
          c_0 & ≜ λ (s : S_0).s, \\
          ρ_0 & ≜ λ (s : S_0)(g : X → X)(a : X).s\,a, &
          r_0 & ≜ δ_j\,m\,X.
\end{align*}\markdownRendererInterblockSeparator
{}Pack them into \markdownRendererInlineMath{p_0 ≜ ⟨…⟩}.\markdownRendererInterblockSeparator
{}Then, given a structure for \markdownRendererInlineMath{q},\markdownRendererSoftLineBreak
{}the structure for \markdownRendererInlineMath{q + 1} is as follows:\begin{align*}
  S^+ & ≜ X → S, \\
          K^+ & ≜ λ (a : X)(b : X).Ka, &
          z^+ & ≜ K, \\
          L^+ & ≜ λ (t : X → S^+)(a : X)(b : X).
              L\,(t\,b)\,a, &
          c^+ & ≜ L, \\
          ρ^+ & ≜ λ (t : S^+)(g : X → X)(a : X).
              ρ \,(ta)\,g\,(ga), &
          r^+ & ≜ ρ \,(m\,S\,c\,z).
\end{align*}\markdownRendererInterblockSeparator
{}A transition term can be written like this:\begin{mathpar}
  T ≜ λ (p : P_X)(h : H_X).
      p\,(λ \,S\,K\,z\,L\,c\,ρ \,r.\,
      h\,S^+\,K^+\,z^+\,L^+\,c^+\,ρ^+\,r^+) : P_X → P_X
\end{mathpar}\markdownRendererInterblockSeparator
{}Therefore, the following term computes the modulus\markdownRendererSoftLineBreak
{}with the expected behaviour.\begin{multline*}
  \mathsf{Mod}_j ≜ λ
          (m : ℕ_{j + 1})(n : ℕ_{j + 4})
          (X : \rk{j})(f : X → X)(x : X).\\
              n\,P_X\,T\,p_0\,
              (λ \,\_\,\_\,\_\,\_\,\_\,\_\,r.
                  r\,X\,f\,x).
\end{multline*}\end{proof}\markdownRendererInterblockSeparator
{}\begin{remark}Note that the only difference between uniform and\markdownRendererSoftLineBreak
{}non-uniform representations is that,\markdownRendererSoftLineBreak
{}in non-uniform representations, we are allowed to\markdownRendererSoftLineBreak
{}pick different \markdownRendererInlineMath{N_I} and \markdownRendererInlineMath{N_O} for each function.\end{remark}\markdownRendererInterblockSeparator
{}\begin{corollary}Unary functions in \markdownRendererInlineMath{\mathcal{E}^2} are uniformly\markdownRendererSoftLineBreak
{}representable in \PTSrk{}.\end{corollary}\markdownRendererInterblockSeparator
{}\begin{proof}There is a binary function \markdownRendererInlineMath{\mathsf{ev} : ℕ^2 → ℕ}\markdownRendererSoftLineBreak
{}in \markdownRendererInlineMath{\mathcal{E}^3} which is universal for unary\markdownRendererSoftLineBreak
{}functions in \markdownRendererInlineMath{\mathcal{E}^2} \markdownRendererCite{1}+{}{Section\markdownRendererNbsp{}5,\markdownRendererNbsp{}p.\markdownRendererNbsp{}873,\markdownRendererNbsp{}footnote\markdownRendererNbsp{}5}{Skordev2008}.\markdownRendererSoftLineBreak
{}In other words, for each \markdownRendererInlineMath{f : ℕ → ℕ} in\markdownRendererSoftLineBreak
{}\markdownRendererInlineMath{\mathcal{E}^2}, there is a code \markdownRendererInlineMath{\mathtt{f} ∈ ℕ}\markdownRendererSoftLineBreak
{}such that\begin{mathpar}
  ∀ n, \quad \mathsf{ev}(\mathtt{f}, n) = f(n).
\end{mathpar}\markdownRendererInterblockSeparator
{}And \markdownRendererInlineMath{\mathcal{E}^3} is precisely elementary\markdownRendererSoftLineBreak
{}recursive functions, so by Theorem \markdownRendererLink{\markdownRendererHash{}non-uniform-rep}{\markdownRendererHash{}non-uniform-rep}{#non-uniform-rep}{}\markdownRendererSoftLineBreak
{}there is a cost \markdownRendererInlineMath{r} and a term\markdownRendererSoftLineBreak
{}\markdownRendererInlineMath{E : ℕ_r → ℕ_r → ℕ_0} such that\begin{mathpar}
  ∀ m n, \quad E\,\overline{m}\,\overline{n} =_β \overline{\mathsf{ev}(m, n)}
\end{mathpar}\markdownRendererInterblockSeparator
{}Therefore, there exist \markdownRendererInlineMath{N_I ≔ ℕ_r} and\markdownRendererSoftLineBreak
{}\markdownRendererInlineMath{N_O ≔ ℕ_0} such that, for each unary\markdownRendererSoftLineBreak
{}\markdownRendererInlineMath{f ∈ \mathcal{E}^2}, there is a term\markdownRendererSoftLineBreak
{}\markdownRendererInlineMath{F ≔ E\,\overline{\mathtt{f}}} such that\begin{mathpar}
  ∀ n, \quad F\,\overline{n} =_β E\,\overline{\mathtt{f}}\,\overline{n} =_β \overline{\mathsf{ev}(\mathtt{f}, n)} = \overline{f(n)}.
\end{mathpar}\end{proof}\markdownRendererInterblockSeparator
{}\begin{remark}While our system appears to be restricted to\markdownRendererSoftLineBreak
{}elementary recursive functions,\markdownRendererSoftLineBreak
{}it turns out that there is another dependently\markdownRendererSoftLineBreak
{}typed system in which functions on natural numbers are\markdownRendererSoftLineBreak
{}\markdownRendererEmphasis{primitive} recursive \markdownRendererCite{1}+{}{}{buchholtz2024primitive}.\markdownRendererSoftLineBreak
{}Interestingly enough, part of their solution is in\markdownRendererSoftLineBreak
{}forbidding function types from appearing in the\markdownRendererSoftLineBreak
{}lowest universe, just like in \PTSrk{}.\end{remark}\markdownRendererInterblockSeparator
{}There is also a boundary result stemming from\markdownRendererSoftLineBreak
{}information-theoretical considerations.\markdownRendererInterblockSeparator
{}\begin{theorem}\label{rep-bound}If a function \markdownRendererInlineMath{f : ℕ^k → ℕ} is representable in\markdownRendererSoftLineBreak
{}\PTSrk{}, it is asymptotically bounded above by an\markdownRendererSoftLineBreak
{}elementary recursive function of the size of the input\markdownRendererSoftLineBreak
{}encoding.\end{theorem}\markdownRendererInterblockSeparator
{}\begin{proof}Since it is representable, there are\markdownRendererSoftLineBreak
{}\markdownRendererInlineMath{N_I : \rk{r}}, \markdownRendererInlineMath{N_O : \rk{s}} and\markdownRendererSoftLineBreak
{}\markdownRendererInlineMath{F : N_I → … → N_I → N_O} such that \markdownRendererInlineMath{F} represents\markdownRendererSoftLineBreak
{}\markdownRendererInlineMath{f}.\markdownRendererSoftLineBreak
{}Therefore,\begin{mathpar}
  ∀ (n_1, …, n_k), \quad
  F\,[n_1]^I\,…\,[n_k]^I =_β [f(n_1, …, n_k)]^O
\end{mathpar}\markdownRendererInterblockSeparator
{}Since terms are definitionally equal,\markdownRendererSoftLineBreak
{}they must have the same normal form which is then\markdownRendererSoftLineBreak
{}equal to \markdownRendererInlineMath{[f(n_1, …, n_k)]^O} as a word.\markdownRendererSoftLineBreak
{}Looking at the proof of Theorem \markdownRendererLink{\markdownRendererHash{}bounded-sn}{\markdownRendererHash{}bounded-sn}{#bounded-sn}{},\markdownRendererSoftLineBreak
{}its length is then bounded above by\markdownRendererSoftLineBreak
{}\markdownRendererInlineMath{\mathcal{O}(c ↑ … ↑ c ↑ i(n_1, …, n_k))} where\markdownRendererSoftLineBreak
{}\markdownRendererInlineMath{i(n_1, …, n_k)} is the length of\markdownRendererSoftLineBreak
{}\markdownRendererInlineMath{F\,[n_1]^I\,…\,[n_k]^I} as a word.\markdownRendererSoftLineBreak
{}The height of the tower of exponents depends not on\markdownRendererSoftLineBreak
{}the \markdownRendererInlineMath{n_i}-s, but on the \markdownRendererInlineMath{d}-bound which is equal\markdownRendererSoftLineBreak
{}to \markdownRendererInlineMath{r + 1} so this bound is \textsc{ELEMENTARY}.\markdownRendererInterblockSeparator
{}By combinatorial considerations,\markdownRendererSoftLineBreak
{}length of \markdownRendererInlineMath{[x]^O} is \markdownRendererInlineMath{Ω(\log{x})} so\begin{multline*}
  f(n_1, …, n_k) = \mathcal{O}(e ↑ \left(c
  ↑ … ↑ c ↑ i(n_1, …,
           n_k)\right)) = \\ = \mathcal{O}(e ↑ c ↑ … ↑ c
  ↑ i(n_1, …, n_k)),
\end{multline*}\markdownRendererInterblockSeparator
{}thus bounded by an elementary recursive function\markdownRendererSoftLineBreak
{}of the size of input encoding.\end{proof}\markdownRendererInterblockSeparator
{}In a particular case where elementary recursive\markdownRendererSoftLineBreak
{}encoding is used, function representable in\markdownRendererSoftLineBreak
{}\PTSrk{} is bounded by elementary recursive\markdownRendererSoftLineBreak
{}function.\markdownRendererInterblockSeparator
{}\begin{corollary}Functions representable in \PTSrk{} with\markdownRendererSoftLineBreak
{}elementary recursive bound on input encoding size\markdownRendererSoftLineBreak
{}are precisely elementary recursive functions.\end{corollary}\markdownRendererInterblockSeparator
{}\begin{proof}Since functions computable under elementary\markdownRendererSoftLineBreak
{}recursive time bound are elementary recursive \markdownRendererCite{1}+{}{}{CookKapron2017PrimitiveRecursive},\markdownRendererSoftLineBreak
{}this follows from Theorems \markdownRendererLink{\markdownRendererHash{}non-uniform-rep}{\markdownRendererHash{}non-uniform-rep}{#non-uniform-rep}{},\markdownRendererSoftLineBreak
{}\markdownRendererLink{\markdownRendererHash{}rep-bound}{\markdownRendererHash{}rep-bound}{#rep-bound}{}.\end{proof}\markdownRendererInterblockSeparator
{}Two questions remain open: a) whether there are\markdownRendererSoftLineBreak
{}non-elementary functions which are non-uniformly\markdownRendererSoftLineBreak
{}computable in \PTSrk{} because of non-elementary\markdownRendererSoftLineBreak
{}input encoding and b) whether there are any\markdownRendererSoftLineBreak
{}functions which are not in \markdownRendererInlineMath{\mathcal{E}^2} and are\markdownRendererSoftLineBreak
{}uniformly computable in \PTSrk{}.\markdownRendererInterblockSeparator
{}
\markdownRendererSectionEnd \markdownRendererSectionBegin
\markdownRendererHeaderAttributeContextBegin
\markdownRendererAttributeIdentifier{rk-model}\markdownRendererHeadingOne{Set-theoretical model of \PTSrk{}}\markdownRendererHeaderAttributeContextEnd{}\markdownRendererInterblockSeparator
{}Now, studying computational properties of a\markdownRendererSoftLineBreak
{}calculus is great, but logic is only useful when\markdownRendererSoftLineBreak
{}it can tell right from wrong.\markdownRendererSoftLineBreak
{}In this section, we are going to provide\markdownRendererSoftLineBreak
{}nontrivial set-theoretical semantics for \PTSrk{}\markdownRendererSoftLineBreak
{}with respect to which it is sound,\markdownRendererSoftLineBreak
{}and where \markdownRendererInlineMath{⟦Π (X : \rk{j}).X⟧} is going to be an\markdownRendererSoftLineBreak
{}empty set, as expected.\markdownRendererParagraphSeparator
{}However, we are going to need an unusual assumption:\markdownRendererInterblockSeparator
{}\begin{theorem}\PTSrk{} has a sound model in ZF+GCH. \formalized{rocq-ai:ZFC_model/ModelSoundness.v\#L133}\end{theorem}\markdownRendererInterblockSeparator
{}\begin{proof}Produced in discussion with GPT-6 Astra.\markdownRendererInterblockSeparator
{}Take distinct codes \markdownRendererInlineMath{c_κ} for every cardinal \markdownRendererInlineMath{κ}\markdownRendererSoftLineBreak
{}below \markdownRendererInlineMath{ℵ_ω} and distinct additional codes \markdownRendererInlineMath{s_j}\markdownRendererSoftLineBreak
{}for \markdownRendererInlineMath{j ∈ ℕ}.\markdownRendererSoftLineBreak
{}Define\begin{mathpar}
  U_i = \{c_κ : κ ≤ ℵ_i\} ∪ \{s_j : j ∈ ℕ\},
  \qquad
  \operatorname{El}(c_κ ) = κ, \quad
  \operatorname{El}(s_j) = U_j.
\end{mathpar}\markdownRendererInterblockSeparator
{}That is,\begin{align*}
  U_0 & = \{c_0, c_1, …, c_ω = c_{ℵ_0},
                  s_0, s_1, …\}, \\
          U_1 & = \{c_0, c_1, …, c_{ℵ_0},
                  c_{ℵ_1}, s_0, s_1, …\}, \\
          …\\
          U_k & = \{
              c_0, c_1, …, c_{ℵ_0},
              c_{ℵ_1}, …, c_{ℵ_k},
              s_0, s_1, …\}\\
          …
\end{align*}\markdownRendererParagraphSeparator
{}(Note that \markdownRendererInlineMath{x ∈ U_i ⇒ |\mathrm{El}(x)| ≤ ℵ_i}\markdownRendererSoftLineBreak
{}since all \markdownRendererInlineMath{U_i} are countable.)\markdownRendererInterblockSeparator
{}Interpretation then goes as follows:\markdownRendererInterblockSeparator
{}\markdownRendererUlBegin
\markdownRendererUlItem Interpret \markdownRendererInlineMath{\rk{i}} as \markdownRendererInlineMath{s_i}.\markdownRendererSoftLineBreak
{}The sort axiom holds because\markdownRendererSoftLineBreak
{}\markdownRendererInlineMath{s_i ∈ U_0 = \operatorname{El}(s_0)}.\markdownRendererUlItemEnd 
\markdownRendererUlItem If \markdownRendererInlineMath{a ∈ U_m} and \markdownRendererInlineMath{b(x) ∈ U_n} for every\markdownRendererSoftLineBreak
{}\markdownRendererInlineMath{x ∈ \operatorname{El}(a)}, the ordinary\markdownRendererSoftLineBreak
{}dependent function set has cardinality at most\markdownRendererParagraphSeparator
{}\markdownRendererDisplayMath{\left|∏_{x ∈ \operatorname{El}(a)}
\operatorname{El}(b(x))\right| ≤ ℵ_n^{ℵ_m} ≤ ℵ_{\max(m + 1, n)}.
}\markdownRendererParagraphSeparator
{}(The last inequality follows from GCH since\markdownRendererSoftLineBreak
{}alephs have finite indices.)\markdownRendererParagraphSeparator
{}Thus, for all such \markdownRendererInlineMath{a, b} we can choose a\markdownRendererSoftLineBreak
{}cardinal \markdownRendererInlineMath{κ_{Π (x : a).b(x)}} (whose\markdownRendererSoftLineBreak
{}code lies in \markdownRendererInlineMath{U_{\max(m + 1, n)}}) along with\markdownRendererSoftLineBreak
{}a bijection\markdownRendererParagraphSeparator
{}\markdownRendererDisplayMath{f : \left(∏_{x ∈ \operatorname{El}(a)}
\operatorname{El}(b(x))
\right) ≃ κ_{Π (x : a).b(x)}
}\markdownRendererParagraphSeparator
{}The eligible pairs \markdownRendererInlineMath{(κ, f)} form a\markdownRendererSoftLineBreak
{}set, so ordinary choice (which holds in\markdownRendererSoftLineBreak
{}ZF+GCH \markdownRendererCite{2}+{}{}{sierpinski1947hypothese}+{}{}{specker1954verallgemeinerte}) suffices.\markdownRendererSoftLineBreak
{}Make each choice\markdownRendererSoftLineBreak
{}once for each \markdownRendererInlineMath{(a, b)}, independently of a\markdownRendererSoftLineBreak
{}typing derivation or rank bound.\markdownRendererUlItemEnd 
\markdownRendererUlItem Abstraction and application are modelled by \markdownRendererInlineMath{f}\markdownRendererSoftLineBreak
{}and \markdownRendererInlineMath{f^{-1}} respectively, as defined above.\markdownRendererUlItemEnd 
\markdownRendererUlItem Contexts are interpreted by valuations.\markdownRendererUlItemEnd 
\markdownRendererUlEnd \markdownRendererParagraphSeparator
{}It is easy to check that the typing rules are\markdownRendererSoftLineBreak
{}respected by this interpretation,\markdownRendererSoftLineBreak
{}making it a sound model.\end{proof}\markdownRendererInterblockSeparator
{}A few corollaries:\markdownRendererInterblockSeparator
{}\markdownRendererUlBeginTight
\markdownRendererUlItem Indeed, the product interpreting\markdownRendererSoftLineBreak
{}\markdownRendererInlineMath{Π (Z : \rk{0}).Z} ranges over \markdownRendererInlineMath{U_0},\markdownRendererSoftLineBreak
{}which contains \markdownRendererInlineMath{c_0}, whose decoding is \markdownRendererInlineMath{∅}.\markdownRendererSoftLineBreak
{}That product is therefore empty.\markdownRendererUlItemEnd 
\markdownRendererUlItem A system with rank cumulativity\markdownRendererSoftLineBreak
{}\markdownRendererInlineMath{t : \mathrm{rk}_j ⇒ t : \mathrm{rk}_{j + 1}} is\markdownRendererSoftLineBreak
{}sound with respect to this model;\markdownRendererSoftLineBreak
{}normalization of bounded terms remains\markdownRendererSoftLineBreak
{}elementary as in Theorem \markdownRendererLink{\markdownRendererHash{}bounded-sn}{\markdownRendererHash{}bounded-sn}{#bounded-sn}{}.\markdownRendererUlItemEnd 
\markdownRendererUlEndTight \markdownRendererInterblockSeparator
{}Note that usage of GCH is unavoidable in this\markdownRendererSoftLineBreak
{}model to validate the rules.\markdownRendererSoftLineBreak
{}We could try building a model in ZFC using\markdownRendererSoftLineBreak
{}\markdownRendererInlineMath{\beth}-s instead of \markdownRendererInlineMath{ℵ}-s, but that would break\markdownRendererSoftLineBreak
{}an assumption that \markdownRendererInlineMath{U_i} are countable,\markdownRendererSoftLineBreak
{}which would falsify axioms \markdownRendererInlineMath{\rk{n} : \rk{0}}.\markdownRendererSoftLineBreak
{}That is not a huge problem, however, since ZF+GCH is\markdownRendererSoftLineBreak
{}interpretable in ZF \markdownRendererCite{2}+{}{}{godel1938consistency}+{}{}{godel1940consistency}\markdownRendererSoftLineBreak
{}(in contrast to ZF + inaccessible cardinals which\markdownRendererSoftLineBreak
{}is not \markdownRendererCite{3}+{}{}{kanamori2003higher}+{}{}{kunen1980set}+{}{}{godel1931formal}).\markdownRendererInterblockSeparator
{}
\markdownRendererSectionEnd \markdownRendererSectionBegin
\markdownRendererHeaderAttributeContextBegin
\markdownRendererAttributeIdentifier{proof-system}\markdownRendererHeadingOne{\PTSrk{} as a proof system}\markdownRendererHeaderAttributeContextEnd{}\markdownRendererInterblockSeparator
{}To demonstrate the ability of \PTSrk{} to embed\markdownRendererSoftLineBreak
{}nontrivial logical theories, we construct a fixed\markdownRendererSoftLineBreak
{}interpretation of Tarski–Mostowski–Robinson\markdownRendererSoftLineBreak
{}arithmetic \markdownRendererInlineMath{\R} in \PTSrk{}: Church numerals and\markdownRendererSoftLineBreak
{}certified finite observations first interpret its\markdownRendererSoftLineBreak
{}relational form \markdownRendererInlineMath{\Rt}; definable graphs yield\markdownRendererSoftLineBreak
{}\markdownRendererInlineMath{\R}.\markdownRendererSoftLineBreak
{}We do so because interpretation of \markdownRendererInlineMath{\Rt} is much\markdownRendererSoftLineBreak
{}simpler than direct interpretation of \markdownRendererInlineMath{\R} and\markdownRendererSoftLineBreak
{}already exposes all the main ideas.\markdownRendererSoftLineBreak
{}Direct translation is available in supplementary\markdownRendererSoftLineBreak
{}formalized materials \formalized{rocq-ai:Interpretation_of_R/functional/DirectRTheorem7.v\#L75}.\markdownRendererParagraphSeparator
{}The following proof was obtained in a discussion\markdownRendererSoftLineBreak
{}with GPT-6 Astra.\markdownRendererInterblockSeparator
{}\markdownRendererSectionBegin
\markdownRendererHeadingTwo{Classical arithmetic and the target calculus}\markdownRendererInterblockSeparator
{}Let \markdownRendererInlineMath{\Rt} be the relational presentation of\markdownRendererSoftLineBreak
{}Tarski–Mostowski–Robinson arithmetic in the\markdownRendererSoftLineBreak
{}first-order language with equality\markdownRendererSoftLineBreak
{}\markdownRendererInlineMath{\{0, S, =, ≤, A, M\}}, with ternary predicates\markdownRendererSoftLineBreak
{}\markdownRendererInlineMath{A, M} and numerals \markdownRendererInlineMath{\bar n = S^n0}.\markdownRendererSoftLineBreak
{}Its axioms are the universal closures of the\markdownRendererSoftLineBreak
{}following schemes, indexed by external \markdownRendererInlineMath{m, n ∈ ℕ}:\begin{align}
  A(\bar m, \bar n, x) & ↔ \overline{m + n} = x,
  \tag{$\widetilde R1$}\label{rtilde:add}\\
  M(\bar m, \bar n, x) & ↔ \overline{mn} = x,
  \tag{$\widetilde R2$}\\
  \bar m & ≠ \bar n \quad (m ≠ n),
  \tag{$\widetilde R3$}\\
  x ≤ \bar n & ↔ ⋁_{i = 0}^{n}x = \bar i,
  \tag{$\widetilde R4$}\label{rtilde:bound}\\
  x ≤ \bar n & ∨ \bar n ≤ x.
  \tag{$\widetilde R5$}\label{rtilde:compare}
\end{align}\markdownRendererInterblockSeparator
{}Here \markdownRendererInlineMath{m + n} and \markdownRendererInlineMath{mn} are calculated externally:\markdownRendererSoftLineBreak
{}the language has no binary arithmetic function\markdownRendererSoftLineBreak
{}symbols.\markdownRendererSoftLineBreak
{}Section \markdownRendererLink{\markdownRendererHash{}sec:nonrelational}{\markdownRendererHash{}sec:nonrelational}{#sec:nonrelational}{} relates this theory\markdownRendererSoftLineBreak
{}to the original, functional arithmetic \markdownRendererInlineMath{\R} \markdownRendererCite{1}+{}{}{tarski1953undecidable}.\markdownRendererParagraphSeparator
{}We now begin constructing the model inside\markdownRendererSoftLineBreak
{}\PTSrk{} and proving that the axioms hold in our\markdownRendererSoftLineBreak
{}model.\markdownRendererSoftLineBreak
{}First, here are some preliminaries on the\markdownRendererSoftLineBreak
{}interpretation of first-order logic formulas:\begin{mathpar}
  ⊥ ≜ (Π (Z : \rk{0}).Z) : \rk{1}, \qquad
  ⊤ ≜ \rk{0} → \rk{0} : \rk{1}, \qquad
  \mathsf{id} ≜ λ (Z : \rk{0}).Z : ⊤.
\end{mathpar}\markdownRendererInterblockSeparator
{}Use the fixed-answer proof types\begin{gather*}
  \nc P ≜ P → ⊥, \qquad
      P \ac Q ≜ (P → Q → ⊥) → ⊥, \\
      ⋁_c(P_0, …, P_k) ≜ (P_0 → ⊥) → ⋯ → (P_k → ⊥) → ⊥, \\
      P \iffc Q ≜ (P → Q) \ac (Q → P).
\end{gather*}\markdownRendererInterblockSeparator
{}A type \markdownRendererInlineMath{P} is \markdownRendererEmphasis{stable} when \markdownRendererInlineMath{\nc\nc P → P} is\markdownRendererSoftLineBreak
{}inhabited.\markdownRendererSoftLineBreak
{}Negative types, dependent functions with stable\markdownRendererSoftLineBreak
{}codomain, and these encoded conjunctions and\markdownRendererSoftLineBreak
{}disjunctions are stable.\markdownRendererSoftLineBreak
{}They support the classical logical rules on stable\markdownRendererSoftLineBreak
{}proof types; case elimination is into stable\markdownRendererSoftLineBreak
{}targets.\markdownRendererSoftLineBreak
{}In particular, \markdownRendererInlineMath{P \oc \nc P} has the proof\markdownRendererSoftLineBreak
{}\markdownRendererInlineMath{λ (k : P → ⊥)(l : (P → ⊥) → ⊥).\,l\,k}.\markdownRendererSoftLineBreak
{}We never eliminate \markdownRendererInlineMath{⊥} into an arbitrary positive\markdownRendererSoftLineBreak
{}type.\markdownRendererParagraphSeparator
{}For a carrier \markdownRendererInlineMath{D}, a domain predicate \markdownRendererInlineMath{\Good},\markdownRendererSoftLineBreak
{}and fixed stable atomic interpretations,\markdownRendererSoftLineBreak
{}translate quantifiers by\begin{align*}
  (∀ x\,φ)^∗ & ≜ Π (x : D).\,\Good(x) → φ^∗ (x), \\
      (∃ x\,φ)^∗ & ≜ (Π (x : D).\,\Good(x) → φ^∗(x) → ⊥) → ⊥.
\end{align*}\markdownRendererInterblockSeparator
{}All translated formulas are stable.\markdownRendererSoftLineBreak
{}Once domain closure and atomic congruence are\markdownRendererSoftLineBreak
{}proved, the logical rules and equality\markdownRendererSoftLineBreak
{}substitution extend structurally to formulas on\markdownRendererSoftLineBreak
{}admitted arguments.\markdownRendererSoftLineBreak
{}All interpretations below use this one logical\markdownRendererSoftLineBreak
{}encoding.\markdownRendererParagraphSeparator
{}To lift propositions \markdownRendererInlineMath{P : \rk{a}} to\markdownRendererSoftLineBreak
{}arbitrary higher ranks \markdownRendererInlineMath{k ≥ 2},\markdownRendererSoftLineBreak
{}let\begin{mathpar}
  L_k(P) ≜ (Π (T : \rk{k - 2}).T → T) → P : \rk{k}.
\end{mathpar}\markdownRendererInterblockSeparator
{}Abstraction gives a map from \markdownRendererInlineMath{P} to \markdownRendererInlineMath{L_k(P)};\markdownRendererSoftLineBreak
{}application on the term \markdownRendererInlineMath{(λ \,T\,x.x)} gives a map in\markdownRendererSoftLineBreak
{}the opposite direction.\markdownRendererInterblockSeparator
{}
\markdownRendererSectionEnd \markdownRendererSectionBegin
\markdownRendererHeadingTwo{Unary domain and equality}\markdownRendererInterblockSeparator
{}Use Church numerals with addition and\markdownRendererSoftLineBreak
{}multiplication as a carrier: \markdownRendererInlineMath{D ≜ ℕ_1}.\markdownRendererSoftLineBreak
{}As a reminder, here are two only \markdownRendererEmphasis{constructors} of\markdownRendererSoftLineBreak
{}Church numerals, \markdownRendererInlineMath{\z} and \markdownRendererInlineMath{\suc}:\begin{mathpar}
  \z ≜ λ (X : \rk{1})(f : X → X)(x : X).x \quad
  \suc ≜ λ (n : ℕ_1)\,X\,f\,x.f\,(n\,X\,f\,x).
\end{mathpar}\markdownRendererInterblockSeparator
{}Even if the theory we interpret and the host type\markdownRendererSoftLineBreak
{}system are not strong enough to support induction\markdownRendererSoftLineBreak
{}on Church numerals, we still need to prove some\markdownRendererSoftLineBreak
{}properties by induction.\markdownRendererSoftLineBreak
{}For this, we introduce a special helper\markdownRendererSoftLineBreak
{}proposition as a domain predicate \markdownRendererInlineMath{\Good}.\markdownRendererSoftLineBreak
{}It would state for a given \markdownRendererInlineMath{d : V} that each\markdownRendererSoftLineBreak
{}inductive property \markdownRendererInlineMath{Q} is satisfied for \markdownRendererInlineMath{d}:\begin{multline*}
  \Good(d) ≜ \bigl(Π (Q : D → \rk{10}).\;
      Q(\z) → (Π (u : D).\,Q(u) → Q(\suc u)) → Q(d)\bigr)\\ : \rk{12}.
\end{multline*}\markdownRendererInterblockSeparator
{}Zero and successor preserve it by abstraction.\markdownRendererSoftLineBreak
{}For example,\begin{mathpar}
  \mathsf{good}_0 ≜ λ Q\,q_0\,q_s.
      \,q_0, \qquad
  \mathsf{good}_s\,d\,g ≜ λ Q\,q_0\,q_s.\,
      q_s\,d\,(g\,Q\,q_0\,q_s).
\end{mathpar}\markdownRendererInterblockSeparator
{}Define equality by negative tests:\begin{mathpar}
  E(d, e) ≜ Π (P : D → \rk{1}).\;
  (\nc P(d)) \iffc (\nc P(e)) \quad : \rk{5}.
\end{mathpar}\markdownRendererInterblockSeparator
{}It is reflexive, symmetric and transitive;\markdownRendererSoftLineBreak
{}every \markdownRendererInlineMath{F : D → D} respects it.\markdownRendererSoftLineBreak
{}Equality substitution for \markdownRendererInlineMath{E} itself follows from\markdownRendererSoftLineBreak
{}symmetry and transitivity.\markdownRendererSoftLineBreak
{}Define the arithmetic atoms by\begin{mathpar}
  A^∗ (d, e, f) ≜ E(\Add\,d\,e, f), \qquad
  M^∗ (d, e, f) ≜ E(\Mul\,d\,e, f).
\end{mathpar}\markdownRendererInterblockSeparator
{}They respect \markdownRendererInlineMath{E} in all arguments;\markdownRendererSoftLineBreak
{}axioms \markdownRendererInlineMath{\widetilde R1} and \markdownRendererInlineMath{\widetilde R2} hold\markdownRendererSoftLineBreak
{}definitionally.\markdownRendererInterblockSeparator
{}
\markdownRendererSectionEnd \markdownRendererSectionBegin
\markdownRendererHeadingTwo{Order from successor-closed unary tests}\markdownRendererInterblockSeparator
{}\begin{definition}A \markdownRendererEmphasis{certified unary algebra} is the fixed telescope\begin{align*}
  \mathcal U = {} & (X : \rk{1})\,(z : X)\,
  (s : X → X)\,(δ : X → \rk{4})\\
  & (d_0 : δ z)\,
  (d_s : Π (u : X).\,δ u → δ (su)).
\end{align*}\end{definition}\markdownRendererInterblockSeparator
{}Basically, this telescope contains all the\markdownRendererSoftLineBreak
{}necessary information to “reinterpret” Church\markdownRendererSoftLineBreak
{}numerals into new carrier \markdownRendererInlineMath{X} with a domain\markdownRendererSoftLineBreak
{}predicate \markdownRendererInlineMath{δ}.\markdownRendererParagraphSeparator
{}Writing \markdownRendererInlineMath{\mathcal U → B} abbreviates the\markdownRendererSoftLineBreak
{}successive dependent arrows for these binders;\markdownRendererSoftLineBreak
{}it does not introduce a record type.\markdownRendererSoftLineBreak
{}Put \markdownRendererInlineMath{\ev_{\mathcal U}(d) = d\,X\,s\,z} and\markdownRendererSoftLineBreak
{}\markdownRendererInlineMath{G(d) ≜ (\mathcal U → δ (\ev_{\mathcal U}(d))) : \rk{6}}.\markdownRendererSoftLineBreak
{}For \markdownRendererInlineMath{P : X → \rk{1}}, write\begin{mathpar}
  \Up_{\mathcal U}(P) ≜ Π (u : X).\,δ u → \nc P(u) → \nc P(su).
\end{mathpar}\markdownRendererInterblockSeparator
{}This proposition asserts that a negation of \markdownRendererInlineMath{P} is\markdownRendererSoftLineBreak
{}upwards closed for ‘good’ values in the new domain \markdownRendererInlineMath{X}.\markdownRendererSoftLineBreak
{}Define\begin{align*}
  O(d, e) ≜ \bigl(
  \mathcal U → Π (P : X → \rk{1}).\; &
  \Up_{\mathcal U}(P) → {}\\
  & \nc P(\ev_{\mathcal U}(d)) → \nc P(\ev_{\mathcal U}(e))\bigr) : \rk{6}.
\end{align*}\markdownRendererInterblockSeparator
{}With this definition, preorder laws and equality\markdownRendererSoftLineBreak
{}compatibility can be derived.\markdownRendererInterblockSeparator
{}\begin{lemma}[\formalized{rocq-ai:Interpretation_of_R/relational/DirectRUnary.v}]\label{r-common}The following holds:\markdownRendererInterblockSeparator
{}\markdownRendererUlBegin
\markdownRendererUlItem \markdownRendererInlineMath{G(\z)}, \markdownRendererInlineMath{G(d) → G(\suc d)} and\markdownRendererSoftLineBreak
{}\markdownRendererInlineMath{\Good(d) → G(d)} are inhabited.\markdownRendererUlItemEnd 
\markdownRendererUlItem \markdownRendererInlineMath{O} is reflexive and transitive.\markdownRendererUlItemEnd 
\markdownRendererUlItem \markdownRendererInlineMath{E(d, e) → O(d, e)} and \markdownRendererInlineMath{G(d) → O(d, \suc d)}\markdownRendererSoftLineBreak
{}are inhabited.\markdownRendererUlItemEnd 
\markdownRendererUlItem Arguments of \markdownRendererInlineMath{O} respect \markdownRendererInlineMath{E}.\markdownRendererUlItemEnd 
\markdownRendererUlEnd \end{lemma}\markdownRendererInterblockSeparator
{}\begin{proof}Proof terms for most propositions are constructed\markdownRendererSoftLineBreak
{}in an obvious way.\markdownRendererSoftLineBreak
{}The last proposition holds because \markdownRendererInlineMath{O} is\markdownRendererSoftLineBreak
{}transitive and \markdownRendererInlineMath{E} is symmetric.\end{proof}\markdownRendererInterblockSeparator
{}Consequently, for each external \markdownRendererInlineMath{m ≤ n},\markdownRendererSoftLineBreak
{}there is a closed proof of\markdownRendererSoftLineBreak
{}\markdownRendererInlineMath{O(\overline{m}, \overline{n})}, built by a finite\markdownRendererSoftLineBreak
{}succession of steps.\markdownRendererInterblockSeparator
{}
\markdownRendererSectionEnd \markdownRendererSectionBegin
\markdownRendererHeaderAttributeContextBegin
\markdownRendererAttributeIdentifier{finite-separation}\markdownRendererHeadingTwo{Finite separation}\markdownRendererHeaderAttributeContextEnd{}\markdownRendererInterblockSeparator
{}In this section, we prove that \markdownRendererInlineMath{\widetilde R3} also\markdownRendererSoftLineBreak
{}holds in our interpretation of \markdownRendererInlineMath{\Rt}.\markdownRendererInterblockSeparator
{}\begin{lemma}[\formalized{rocq-ai:Interpretation_of_R/relational/DirectRSeparation.v\#L143}]For each \markdownRendererInlineMath{n ∈ ℕ}, there is a proof of\markdownRendererSoftLineBreak
{}\markdownRendererInlineMath{O(\overline{n + 1}, \overline{n}) → ⊥}.\end{lemma}\markdownRendererInterblockSeparator
{}\begin{proof}By contradiction.\markdownRendererSoftLineBreak
{}Suppose \markdownRendererInlineMath{H : O(\overline{n + 1}, \overline{n})};\markdownRendererSoftLineBreak
{}put \markdownRendererInlineMath{c = n + 1}, \markdownRendererInlineMath{q = c + 1},\markdownRendererSoftLineBreak
{}and remember already familiar type of selectors\begin{mathpar}
  F_q ≜ \underbrace{\rk{1} → ⋯ → \rk{1}}_{q\ \mathrm{arguments}} → \rk{1}, \qquad
  e_i ≜ λ Z_0… Z_{q - 1}.Z_i.
\end{mathpar}\markdownRendererInterblockSeparator
{}We are going to use them as a “model” to interpret\markdownRendererSoftLineBreak
{}an assumed \markdownRendererInlineMath{O} and get a contradiction:\markdownRendererSoftLineBreak
{}set \markdownRendererInlineMath{X = F_q}, \markdownRendererInlineMath{z = e_0}, and use “saturated\markdownRendererSoftLineBreak
{}addition” as a successor:\begin{mathpar}
  s(u) ≜ λ Z_0… Z_{q - 1}.\
  u\,Z_{\min(0 + 1, c)}⋯ Z_{\min(c + 1, c)}.
\end{mathpar}\markdownRendererInterblockSeparator
{}Thus, \markdownRendererInlineMath{s(e_i) =_β e_{\min(i + 1, c)}}.\markdownRendererSoftLineBreak
{}Define the domain predicate as\begin{mathpar}
  δ (u) ≜ \bigl(Π (Q : F_q → \rk{3}).\;
  Q(e_0) → ⋯ → Q(e_c) → Q(u)\bigr) : \rk{4}.
\end{mathpar}\markdownRendererInterblockSeparator
{}Selector projection gives \markdownRendererInlineMath{d_0}.\markdownRendererSoftLineBreak
{}To construct \markdownRendererInlineMath{d_s}, open the output \markdownRendererInlineMath{δ (su)} at\markdownRendererSoftLineBreak
{}\markdownRendererInlineMath{Q} and its branch proofs, and specialize the\markdownRendererSoftLineBreak
{}input certificate to \markdownRendererInlineMath{v ↦ Q(sv)}.\markdownRendererSoftLineBreak
{}Its branches are the supplied proofs for the\markdownRendererSoftLineBreak
{}appropriate successor selectors.\markdownRendererInterblockSeparator
{}Use the unary test\begin{mathpar}
  P(u) ≜ u\,\underbrace{⊤⋯⊤}_{c\ \mathrm{entries}}\,⊥.
\end{mathpar}\markdownRendererInterblockSeparator
{}To prove \markdownRendererInlineMath{\Up(P)}, eliminate \markdownRendererInlineMath{δ (u)} into the\markdownRendererSoftLineBreak
{}rank-three motive \markdownRendererInlineMath{v ↦ \nc P(v) → \nc P(sv)}.\markdownRendererSoftLineBreak
{}At \markdownRendererInlineMath{e_c} both tests are \markdownRendererInlineMath{⊥} and the implication is\markdownRendererSoftLineBreak
{}identity.\markdownRendererSoftLineBreak
{}At \markdownRendererInlineMath{e_i} with \markdownRendererInlineMath{i < c}, there is a premise \markdownRendererInlineMath{⊤ → ⊥}.\markdownRendererInterblockSeparator
{}Specializing \markdownRendererInlineMath{H : O(\overline{c}, \overline{n})}\markdownRendererSoftLineBreak
{}with these data yields \markdownRendererInlineMath{(⊥ → ⊥) → (⊤ → ⊥)}.\end{proof}\markdownRendererInterblockSeparator
{}\begin{lemma}[\formalized{rocq-ai:Interpretation_of_R/relational/DirectRSeparation.v\#L193}]\label{sep}For each \markdownRendererInlineMath{m > n}, there is a proof of\markdownRendererSoftLineBreak
{}\markdownRendererInlineMath{O(\overline{m}, \overline{n}) → ⊥}.\end{lemma}\markdownRendererInterblockSeparator
{}\begin{proof}Suppose \markdownRendererInlineMath{O(\overline{m}, \overline{n})} is\markdownRendererSoftLineBreak
{}inhabited; since \markdownRendererInlineMath{m > n},\markdownRendererSoftLineBreak
{}\markdownRendererInlineMath{O(\overline{n + 1}, \overline{m})} is also\markdownRendererSoftLineBreak
{}inhabited.\markdownRendererSoftLineBreak
{}By transitivity of \markdownRendererInlineMath{O} it means that\markdownRendererSoftLineBreak
{}\markdownRendererInlineMath{O(\overline{n + 1}, \overline{n})} which we know\markdownRendererSoftLineBreak
{}is absurd.\end{proof}\markdownRendererInterblockSeparator
{}Thus, \markdownRendererInlineMath{\widetilde R3} also holds: if \markdownRendererInlineMath{m ≠ n},\markdownRendererSoftLineBreak
{}then either \markdownRendererInlineMath{m > n} or \markdownRendererInlineMath{m < n} (these are external\markdownRendererSoftLineBreak
{}natural numbers).\markdownRendererSoftLineBreak
{}Assuming \markdownRendererInlineMath{E(\overline{m}, \overline{n})},\markdownRendererSoftLineBreak
{}we obtain a paradoxical \markdownRendererInlineMath{O} and derive a\markdownRendererSoftLineBreak
{}contradiction by Lemma \markdownRendererLink{\markdownRendererHash{}sep}{\markdownRendererHash{}sep}{#sep}{}.\markdownRendererInterblockSeparator
{}
\markdownRendererSectionEnd \markdownRendererSectionBegin
\markdownRendererHeaderAttributeContextBegin
\markdownRendererAttributeIdentifier{bounded-cases}\markdownRendererHeadingTwo{Bounded cases}\markdownRendererHeaderAttributeContextEnd{}\markdownRendererInterblockSeparator
{}Here, we prove a crucial lemma from which\markdownRendererSoftLineBreak
{}remaining axioms follow.\markdownRendererSoftLineBreak
{}Set\begin{mathpar}
  \Case_n(d) ≜ ⋁_c
  \bigl(E(d, \overline{0}), …, E(d, \overline{n}),
  O(\overline{n + 1}, d)\bigr) : \rk{8}.
\end{mathpar}\markdownRendererInterblockSeparator
{}\begin{lemma}[\formalized{rocq-ai:Interpretation_of_R/relational/DirectRUnary.v\#L808}]\label{case}\markdownRendererInlineMath{Π (d : D).\;\Good(d) → \Case_n(d)} is inhabited\markdownRendererSoftLineBreak
{}for external \markdownRendererInlineMath{n}.\end{lemma}\markdownRendererInterblockSeparator
{}\begin{proof}We are going to use the following positive\markdownRendererSoftLineBreak
{}invariant for \markdownRendererInlineMath{Q}:\begin{mathpar}
  \Inv_n(d) ≜ Π (Z : \rk{8}).\;(G(d) → \Case_n(d) → Z) → Z \quad : \rk{10}.
\end{mathpar}\markdownRendererInterblockSeparator
{}At zero, supply \markdownRendererInlineMath{G(\z)} and the\markdownRendererSoftLineBreak
{}\markdownRendererInlineMath{E(\z, \overline{0})} branch.\markdownRendererSoftLineBreak
{}At successor, open the target continuation and\markdownRendererSoftLineBreak
{}eliminate the child invariant at that same \markdownRendererInlineMath{Z},\markdownRendererSoftLineBreak
{}obtaining \markdownRendererInlineMath{G(d)} and \markdownRendererInlineMath{\Case_n(d)}.\markdownRendererSoftLineBreak
{}Obtain \markdownRendererInlineMath{G(\suc d)} independently.\markdownRendererSoftLineBreak
{}To construct the new negative case proof,\markdownRendererSoftLineBreak
{}all branch continuations return \markdownRendererInlineMath{⊥}.\markdownRendererSoftLineBreak
{}An \markdownRendererInlineMath{i < n} branch becomes \markdownRendererInlineMath{i + 1} by successor\markdownRendererSoftLineBreak
{}congruence.\markdownRendererSoftLineBreak
{}The \markdownRendererInlineMath{i = n} branch gives\markdownRendererSoftLineBreak
{}\markdownRendererInlineMath{E(\suc d, \overline{n + 1})},\markdownRendererSoftLineBreak
{}hence \markdownRendererInlineMath{O(\overline{n + 1}, \suc d)}.\markdownRendererSoftLineBreak
{}The large branch is composed with \markdownRendererInlineMath{O(d, \suc d)}.\markdownRendererSoftLineBreak
{}This proves successor closure.\markdownRendererInterblockSeparator
{}Instantiate \markdownRendererInlineMath{\Good(d)} with \markdownRendererInlineMath{\Inv_n}.\markdownRendererSoftLineBreak
{}Project its case component by choosing\markdownRendererSoftLineBreak
{}\markdownRendererInlineMath{Z = \Case_n(d)} and continuation \markdownRendererInlineMath{λ g\,c.c}.\markdownRendererSoftLineBreak
{}No negative disjunction is eliminated into\markdownRendererSoftLineBreak
{}positive \markdownRendererInlineMath{G}-evidence.\end{proof}\markdownRendererInterblockSeparator
{}
\markdownRendererSectionEnd \markdownRendererSectionBegin
\markdownRendererHeadingTwo{The relational interpretation}\markdownRendererInterblockSeparator
{}\begin{theorem}[\formalized{rocq-ai:Interpretation_of_R/relational/DirectRTheorem4.v\#L57}]\label{main}\markdownRendererNote{The linked theorem gives original-PTS\markdownRendererSoftLineBreak
{}inhabitants for all axiom instances and all source\markdownRendererSoftLineBreak
{}derivations.\markdownRendererSoftLineBreak
{}The generic constructive witness search turns each\markdownRendererSoftLineBreak
{}such existence proof into a proof term together\markdownRendererSoftLineBreak
{}with its typing derivation.}\markdownRendererSoftLineBreak
{}The carrier \markdownRendererInlineMath{D}, admitted domain \markdownRendererInlineMath{\Good},\markdownRendererSoftLineBreak
{}closed operations \markdownRendererInlineMath{\z, \suc},\markdownRendererSoftLineBreak
{}and relations \markdownRendererInlineMath{E, O, A^∗, M^∗} interpret the\markdownRendererSoftLineBreak
{}classical relational theory \markdownRendererInlineMath{\Rt} in \PTSrk{}.\markdownRendererSoftLineBreak
{}Every translated axiom has an effectively\markdownRendererSoftLineBreak
{}constructible closed proof term.\markdownRendererSoftLineBreak
{}The interpretation is closed under classical\markdownRendererSoftLineBreak
{}first-order derivations: every \markdownRendererInlineMath{\Rt}-derivation of\markdownRendererSoftLineBreak
{}a sentence effectively yields a closed \PTSrk{}\markdownRendererSoftLineBreak
{}proof term of its translation.\markdownRendererSoftLineBreak
{}True ground arithmetic atoms have proofs,\markdownRendererSoftLineBreak
{}and false ground atoms have internal refutations.\end{theorem}\markdownRendererInterblockSeparator
{}\begin{proof}Domain closure, atomic stability,\markdownRendererSoftLineBreak
{}and congruence have been established;\markdownRendererSoftLineBreak
{}hence source deductions translate to PTS proofs.\markdownRendererSoftLineBreak
{}We know that \markdownRendererInlineMath{\widetilde R1}-\markdownRendererInlineMath{\widetilde R3} hold;\markdownRendererSoftLineBreak
{}let us check the remaining axioms.\markdownRendererParagraphSeparator
{}For bounded order (\markdownRendererInlineMath{\widetilde R4}),\markdownRendererSoftLineBreak
{}apply Lemma \markdownRendererLink{\markdownRendererHash{}case}{\markdownRendererHash{}case}{#case}{}.\markdownRendererSoftLineBreak
{}Under \markdownRendererInlineMath{O(d, \overline{n})}, a case involving\markdownRendererSoftLineBreak
{}\markdownRendererInlineMath{O(\overline{n + 1}, d)} contradicts\markdownRendererSoftLineBreak
{}\markdownRendererInlineMath{\nc O(\overline{n + 1}, \overline{n})} by\markdownRendererSoftLineBreak
{}transitivity.\markdownRendererSoftLineBreak
{}Conversely, a \markdownRendererInlineMath{E(d, \overline{i})} branch with\markdownRendererSoftLineBreak
{}\markdownRendererInlineMath{i ≤ n} yields \markdownRendererInlineMath{O(d, \overline{n})} via \markdownRendererInlineMath{E ⇒ O}\markdownRendererSoftLineBreak
{}and ground comparison.\markdownRendererParagraphSeparator
{}For comparability (\markdownRendererInlineMath{\widetilde R5}),\markdownRendererSoftLineBreak
{}the finite branches of Lemma \markdownRendererLink{\markdownRendererHash{}case}{\markdownRendererHash{}case}{#case}{} give\markdownRendererSoftLineBreak
{}\markdownRendererInlineMath{O(d, \overline{n})} and the large branch gives\markdownRendererSoftLineBreak
{}\markdownRendererInlineMath{O(\overline{n}, d)}.\markdownRendererParagraphSeparator
{}The ground proof/refutation claims follow from the\markdownRendererSoftLineBreak
{}arithmetic beta equations, reflexivity,\markdownRendererSoftLineBreak
{}ground order, and Lemma \markdownRendererLink{\markdownRendererHash{}sep}{\markdownRendererHash{}sep}{#sep}{}.\end{proof}\markdownRendererInterblockSeparator
{}\begin{remark}Constructions from Sections \markdownRendererLink{\markdownRendererHash{}finite-separation}{\markdownRendererHash{}finite-separation}{#finite-separation}{}\markdownRendererSoftLineBreak
{}and \markdownRendererLink{\markdownRendererHash{}bounded-cases}{\markdownRendererHash{}bounded-cases}{#bounded-cases}{} generalise to arbitrary\markdownRendererSoftLineBreak
{}Church-encoded inductive types providing\markdownRendererSoftLineBreak
{}constructor disjointness and injectivity,\markdownRendererSoftLineBreak
{}as well as finite case analysis for subtrees of\markdownRendererSoftLineBreak
{}fixed constructor trees.\markdownRendererSoftLineBreak
{}These results can be seen as internal inversion\markdownRendererSoftLineBreak
{}lemmas.\end{remark}\markdownRendererInterblockSeparator
{}While this interpretation is full, it, however,\markdownRendererSoftLineBreak
{}is not faithful on arbitrary first-order\markdownRendererSoftLineBreak
{}sentences.\markdownRendererInterblockSeparator
{}\begin{remark}[Failure of faithfulness \formalized{rocq-ai:Interpretation_of_R/relational/DirectRTheorem4.v\#L47}]The interpretation of \markdownRendererInlineMath{\Rt} is not faithful.\end{remark}\markdownRendererInterblockSeparator
{}\begin{proof}The translation of \markdownRendererInlineMath{ρ ≜ ∀ x\,(x ≤ x)} is provable\markdownRendererSoftLineBreak
{}in \PTSrk{}, but \markdownRendererInlineMath{\Rt ⊬ ρ}.\markdownRendererParagraphSeparator
{}We already know that \markdownRendererInlineMath{O} is reflexive (Lemma \markdownRendererLink{\markdownRendererHash{}r-common}{\markdownRendererHash{}r-common}{#r-common}{}).\markdownRendererSoftLineBreak
{}But why is reflexivity not derivable in\markdownRendererSoftLineBreak
{}\markdownRendererInlineMath{\widetilde R}?\markdownRendererSoftLineBreak
{}It follows from model-theoretic considerations:\markdownRendererSoftLineBreak
{}there is a set-theoretic countermodel.\markdownRendererInterblockSeparator
{}Take \markdownRendererInlineMath{N^+ = ℕ ⊔ \{ω\}}, where \markdownRendererInlineMath{ω} is one new\markdownRendererSoftLineBreak
{}element.\markdownRendererSoftLineBreak
{}Interpret \markdownRendererInlineMath{0} and successor as usual on \markdownRendererInlineMath{ℕ},\markdownRendererSoftLineBreak
{}put \markdownRendererInlineMath{Sω = ω}, and use actual equality.\markdownRendererSoftLineBreak
{}For \markdownRendererInlineMath{x, y ∈ ℕ}, let \markdownRendererInlineMath{A(x, y, z)} hold exactly when\markdownRendererSoftLineBreak
{}\markdownRendererInlineMath{z = x + y}, and let \markdownRendererInlineMath{M(x, y, z)} hold exactly\markdownRendererSoftLineBreak
{}when \markdownRendererInlineMath{z = xy}.\markdownRendererSoftLineBreak
{}When at least one of \markdownRendererInlineMath{x, y} is \markdownRendererInlineMath{ω},\markdownRendererSoftLineBreak
{}let each predicate hold exactly when \markdownRendererInlineMath{z = 0}.\markdownRendererSoftLineBreak
{}Define\begin{mathpar}
  x ≤^+ y \quad ⟺ \quad
  \bigl(x, y ∈ ℕ\ \text{and}
  \ x ≤_{ℕ} y\bigr)
  \ \text{or}\ \bigl(x ∈ ℕ\ \text{and}
  \ y = ω \bigr).
\end{mathpar}\markdownRendererParagraphSeparator
{}The ground graph schemes hold for every third\markdownRendererSoftLineBreak
{}argument, and distinct numerals are unequal.\markdownRendererSoftLineBreak
{}For every standard \markdownRendererInlineMath{n}, the downset of \markdownRendererInlineMath{n} is\markdownRendererSoftLineBreak
{}exactly \markdownRendererInlineMath{\{0, …, n\}}, proving both directions of\markdownRendererSoftLineBreak
{}(\markdownRendererLink{\markdownRendererHash{}rtilde:bound}{\markdownRendererHash{}rtilde:bound}{#rtilde:bound}{}).\markdownRendererSoftLineBreak
{}Every standard number is comparable with \markdownRendererInlineMath{n},\markdownRendererSoftLineBreak
{}and \markdownRendererInlineMath{n ≤^+ ω}, proving (\markdownRendererLink{\markdownRendererHash{}rtilde:compare}{\markdownRendererHash{}rtilde:compare}{#rtilde:compare}{}).\markdownRendererSoftLineBreak
{}Hence, this is a model of \markdownRendererInlineMath{\Rt}.\markdownRendererParagraphSeparator
{}However, \markdownRendererInlineMath{ω \not ≤^+ ω}, so \markdownRendererInlineMath{ρ} is false.\end{proof}\markdownRendererInterblockSeparator
{}Actually, it is possible to prove inside \PTSrk{}\markdownRendererSoftLineBreak
{}that \markdownRendererInlineMath{\Good} Church numerals form a \markdownRendererEmphasis{strictly ordered\markdownRendererSoftLineBreak
{}commutative unital semiring}.\markdownRendererSoftLineBreak
{}A generated proof is available in\markdownRendererSoftLineBreak
{}supplementary materials.\markdownRendererParagraphSeparator
{}
\markdownRendererSectionEnd \markdownRendererSectionBegin
\markdownRendererHeaderAttributeContextBegin
\markdownRendererAttributeIdentifier{sec:nonrelational}\markdownRendererHeadingTwo{Relation to the original presentation of \texorpdfstring{$\R$}{R}}\markdownRendererHeaderAttributeContextEnd{}\markdownRendererInterblockSeparator
{}The original Tarski–Mostowski–Robinson\markdownRendererSoftLineBreak
{}arithmetic \markdownRendererInlineMath{\R} \markdownRendererCite{1}+{}{}{tarski1953undecidable} uses\markdownRendererSoftLineBreak
{}\markdownRendererInlineMath{\{0, S,+, ⋅, ≤ \}} and the following schemes,\markdownRendererSoftLineBreak
{}again with universal closure:\begin{align}
  \bar m + \bar n & = \overline{m + n}, & & \tag{R1}\\
  \bar m ⋅ \bar n & = \overline{mn}, & & \tag{R2}\\
  \bar m & ≠ \bar n \quad (m ≠ n), & & \tag{R3}\\
  x ≤ \bar n & ⟶
      ⋁_{i = 0}^{n}x = \bar i,
      & & \tag{R4}\label{r:bound}\\
  x ≤ \bar n & ∨ \bar n ≤ x.
      & & \tag{R5}\label{r:compare}
\end{align}\markdownRendererInterblockSeparator
{}Requiring \markdownRendererInlineMath{A, M} in \markdownRendererInlineMath{\Rt} to be total,\markdownRendererSoftLineBreak
{}single-valued graphs makes them interchangeable\markdownRendererSoftLineBreak
{}with the function symbols \markdownRendererInlineMath{+, ⋅}.\markdownRendererSoftLineBreak
{}This functional strengthening is definitionally\markdownRendererSoftLineBreak
{}equivalent to \markdownRendererInlineMath{\R}; even without that\markdownRendererSoftLineBreak
{}strengthening, the two theories have the same\markdownRendererSoftLineBreak
{}interpretability degree.\markdownRendererInterblockSeparator
{}\begin{proposition}The classical theories \markdownRendererInlineMath{\Rt} and \markdownRendererInlineMath{\R} are mutually\markdownRendererSoftLineBreak
{}interpretable \markdownRendererNote{Mutual interpretability does not\markdownRendererSoftLineBreak
{}imply bi-interpretability.\markdownRendererSoftLineBreak
{}In particular, the totalized graphs need not\markdownRendererSoftLineBreak
{}coincide with the original relations \markdownRendererInlineMath{A, M} on\markdownRendererSoftLineBreak
{}arbitrary arguments.}.\end{proposition}\markdownRendererInterblockSeparator
{}\begin{proof}In \markdownRendererInlineMath{\R}, interpret \markdownRendererInlineMath{A(x, y, z)} by \markdownRendererInlineMath{x + y = z} and\markdownRendererSoftLineBreak
{}\markdownRendererInlineMath{M(x, y, z)} by \markdownRendererInlineMath{xy = z}.\markdownRendererSoftLineBreak
{}The ground arithmetic equations give the two graph\markdownRendererSoftLineBreak
{}schemes.\markdownRendererSoftLineBreak
{}For \markdownRendererInlineMath{i < j}, an assumption \markdownRendererInlineMath{\bar j ≤ \bar i}\markdownRendererSoftLineBreak
{}contradicts (\markdownRendererLink{\markdownRendererHash{}r:bound}{\markdownRendererHash{}r:bound}{#r:bound}{}) and numeral separation;\markdownRendererSoftLineBreak
{}(\markdownRendererLink{\markdownRendererHash{}r:compare}{\markdownRendererHash{}r:compare}{#r:compare}{}) therefore yields \markdownRendererInlineMath{\bar i ≤ \bar j}.\markdownRendererSoftLineBreak
{}The case \markdownRendererInlineMath{i = j} follows from (\markdownRendererLink{\markdownRendererHash{}r:compare}{\markdownRendererHash{}r:compare}{#r:compare}{})\markdownRendererSoftLineBreak
{}directly.\markdownRendererSoftLineBreak
{}Thus, \markdownRendererInlineMath{\R} proves all true ground comparisons,\markdownRendererSoftLineBreak
{}and equality substitution supplies the reverse\markdownRendererSoftLineBreak
{}implication in (\markdownRendererLink{\markdownRendererHash{}rtilde:bound}{\markdownRendererHash{}rtilde:bound}{#rtilde:bound}{}).\markdownRendererSoftLineBreak
{}The remaining relational schemes are immediate.\markdownRendererInterblockSeparator
{}For the other direction define, in \markdownRendererInlineMath{\Rt},\begin{align*}
  U_A(x, y, z) & ≜ A(x, y, z) ∧ ∀ w\bigl(A(x, y, w) → w = z\bigr), \\
          F_A(x, y, z) & ≜ U_A(x, y, z) ∨ \bigl(z = 0 ∧ ¬ ∃ v\,U_A(x, y, v)\bigr).
\end{align*}\markdownRendererParagraphSeparator
{}For fixed \markdownRendererInlineMath{x, y}, either a unique \markdownRendererInlineMath{A}-output\markdownRendererSoftLineBreak
{}exists or it does not.\markdownRendererSoftLineBreak
{}In the first case \markdownRendererInlineMath{F_A} selects that output;\markdownRendererSoftLineBreak
{}in the second it selects zero.\markdownRendererSoftLineBreak
{}Hence, \markdownRendererInlineMath{F_A} is provably total and single-valued.\markdownRendererSoftLineBreak
{}At standard inputs, (\markdownRendererLink{\markdownRendererHash{}rtilde:add}{\markdownRendererHash{}rtilde:add}{#rtilde:add}{}) forces its\markdownRendererSoftLineBreak
{}value to be \markdownRendererInlineMath{\overline{m + n}}.\markdownRendererSoftLineBreak
{}Define \markdownRendererInlineMath{F_M} identically with \markdownRendererInlineMath{M} in place of \markdownRendererInlineMath{A};\markdownRendererSoftLineBreak
{}its standard value is \markdownRendererInlineMath{\overline{mn}}.\markdownRendererSoftLineBreak
{}These graphs interpret \markdownRendererInlineMath{+, ⋅};\markdownRendererSoftLineBreak
{}keep \markdownRendererInlineMath{0, S, =, ≤} unchanged.\markdownRendererSoftLineBreak
{}All five schemes of \markdownRendererInlineMath{\R} follow.\markdownRendererParagraphSeparator
{}Composing this graph interpretation with Theorem \markdownRendererLink{\markdownRendererHash{}main}{\markdownRendererHash{}main}{#main}{}\markdownRendererSoftLineBreak
{}gives the claimed interpretation in \PTSrk{}.\markdownRendererSoftLineBreak
{}The composition uses the encoded classical\markdownRendererSoftLineBreak
{}quantifiers to express values and uniqueness;\markdownRendererSoftLineBreak
{}it does not assert \markdownRendererInlineMath{\Good}-closure of the raw\markdownRendererSoftLineBreak
{}Church operators \markdownRendererInlineMath{\Add, \Mul}.\end{proof}\markdownRendererInterblockSeparator
{}\begin{corollary}\PTSrk{} interprets \markdownRendererInlineMath{\R} by definable arithmetic\markdownRendererSoftLineBreak
{}graphs.\end{corollary}\markdownRendererInterblockSeparator
{}
\markdownRendererSectionEnd 
\markdownRendererSectionEnd \markdownRendererSectionBegin
\markdownRendererHeadingOne{Conclusion}\markdownRendererInterblockSeparator
{}To sum it all up, what started as a funny idea\markdownRendererSoftLineBreak
{}(what if there actually was a sound pure type\markdownRendererSoftLineBreak
{}system with \markdownRendererInlineMath{\mathsf{type} : \mathsf{type}}) led\markdownRendererSoftLineBreak
{}us to stumble on a truly unique dependently typed\markdownRendererSoftLineBreak
{}pure type system, \PTSrk{}.\markdownRendererParagraphSeparator
{}First of all, its normalization proof is\markdownRendererSoftLineBreak
{}suspiciously simple and reminds of classic proofs\markdownRendererSoftLineBreak
{}of normalization for STLC where,\markdownRendererSoftLineBreak
{}instead of an established technique by\markdownRendererSoftLineBreak
{}reducibility candidates, one patiently tracks\markdownRendererSoftLineBreak
{}complexities of arrow types and makes sure that\markdownRendererSoftLineBreak
{}they become simpler after sequences of reductions\markdownRendererSoftLineBreak
{}(Section \markdownRendererLink{\markdownRendererHash{}rk-sn}{\markdownRendererHash{}rk-sn}{#rk-sn}{}).\markdownRendererSoftLineBreak
{}To make sure that this approach works for our\markdownRendererSoftLineBreak
{}system, we formalized our proof in Rocq proof\markdownRendererSoftLineBreak
{}assistant (Section \markdownRendererLink{\markdownRendererHash{}formalization}{\markdownRendererHash{}formalization}{#formalization}{}) and in doing\markdownRendererSoftLineBreak
{}so prepared the groundwork for future verified\markdownRendererSoftLineBreak
{}interactions with \PTSrk{}.\markdownRendererParagraphSeparator
{}After noticing that normalization proof yields a\markdownRendererSoftLineBreak
{}primitive recursive normalization procedure,\markdownRendererSoftLineBreak
{}we found tighter bounds and found out that it has\markdownRendererSoftLineBreak
{}the same complexity as normalization of STLC,\markdownRendererSoftLineBreak
{}which was previously unheard of for a dependently\markdownRendererSoftLineBreak
{}typed system (Section \markdownRendererLink{\markdownRendererHash{}complexity}{\markdownRendererHash{}complexity}{#complexity}{}).\markdownRendererSoftLineBreak
{}In contrast with existing pure type systems,\markdownRendererSoftLineBreak
{}our system is able to represent elementary\markdownRendererSoftLineBreak
{}recursive functions (both in the sense that it is\markdownRendererSoftLineBreak
{}stronger than STLC and weaker than more classical\markdownRendererSoftLineBreak
{}dependently typed systems) and seems intimately\markdownRendererSoftLineBreak
{}connected to it (Section \markdownRendererLink{\markdownRendererHash{}computational}{\markdownRendererHash{}computational}{#computational}{}).\markdownRendererSoftLineBreak
{}To check if constructions inside \PTSrk{} even\markdownRendererSoftLineBreak
{}make sense, we constructed a set-theoretical model\markdownRendererSoftLineBreak
{}for it which assumes a Generalized Continuum\markdownRendererSoftLineBreak
{}Hypothesis to strictify the hierarchy of\markdownRendererSoftLineBreak
{}\markdownRendererInlineMath{ℵ}-s which act as interpretations of our function\markdownRendererSoftLineBreak
{}types (Section \markdownRendererLink{\markdownRendererHash{}rk-model}{\markdownRendererHash{}rk-model}{#rk-model}{}).\markdownRendererSoftLineBreak
{}Finally, we provide a lower bound on\markdownRendererSoftLineBreak
{}proof-theoretical strength of \PTSrk{} by\markdownRendererSoftLineBreak
{}providing a full interpretation of a relational\markdownRendererSoftLineBreak
{}version of Tarski–Mostowski–Robinson arithmetic\markdownRendererSoftLineBreak
{}\markdownRendererInlineMath{R} inside it (Section \markdownRendererLink{\markdownRendererHash{}proof-system}{\markdownRendererHash{}proof-system}{#proof-system}{});\markdownRendererSoftLineBreak
{}we did it partly because we suspect our system to\markdownRendererSoftLineBreak
{}be able to prove its own consistency in the sense\markdownRendererSoftLineBreak
{}Pakhomov finitistic set theory does \markdownRendererCite{1}+{}{}{Pakhomov2019AWeakSet},\markdownRendererSoftLineBreak
{}therefore, interpretation of \markdownRendererInlineMath{R} is the natural\markdownRendererSoftLineBreak
{}first step.\markdownRendererParagraphSeparator
{}\printbibliography{}
\markdownRendererSectionEnd \markdownRendererDocumentEnd
\endgroup
\end{document}